\documentclass[12pt,reqno]{amsart}
\usepackage{amsthm}
\usepackage{amssymb}
\usepackage{amsmath}
\usepackage{mathrsfs}
\usepackage{amsfonts}
\usepackage{amssymb,amsmath}
 \usepackage[colorlinks, linkcolor=black, citecolor=blue,   hypertexnames=false]{hyperref}
 \usepackage[numbers,sort&compress]{natbib}
 \usepackage{fullpage}

\numberwithin{equation}{section}
\newtheorem{thm}{Theorem}[section]

\newtheorem{lem}[thm]{Lemma}
\newtheorem{cor}[thm]{Corollary}
\newtheorem{rmk}[thm]{Remark}

\allowdisplaybreaks[4]

\newcommand{\al}{\alpha}
 
 \newcommand{\ld}{\lambda}
 
 \newcommand{\de}{\delta}
 \newcommand{\De}{\Delta}
 \newcommand{\ep}{\varepsilon}

 \newcommand{\ka}{\kappa}

 \newcommand{\R}{\mathbb{R}}

 \newcommand{\p}{\partial}
 \newcommand{\n}{\nabla}

\def\XXint#1#2#3{{\setbox0=\hbox{$#1{#2#3}{\int}$}
     \vcenter{\hbox{$#2#3$}}\kern-.5\wd0}}

 \def\<{\langle} \def\>{\rangle}
 \def\({\left(} \def\){\right)}

\title{Gradient estimate and Liouville theorem for a semilinear parabolic equation with variable coefficient}

\author[C. Song]{Chong Song}
\address{School of Mathematical Sciences, Xiamen University
\newline\indent
Xiamen, Fujian 361005, P.R. China
}
\email{songchong@xmu.edu.cn}

\author[J. Wu]{Jibo Wu}
\address{School of Mathematical Sciences, Xiamen University
\newline\indent
Xiamen, Fujian 361005, P.R. China
}
\email{wujibo@stu.xmu.edu.cn}

\date{August 12, 2026}

\keywords{semilinear parabolic equation, gradient estimate, Liouville theorem, Nash-Moser iteration}

\begin{document}

\begin{abstract}
In this paper, we investigate the semilinear parabolic equation 
\[ \partial_tf-\Delta f=a(x,t)f^p \]
on a complete Riemannian manifold with Ricci curvature bounded from below. By means of Nash-Moser iteration, we establish Li–Yau-type gradient estimates for positive solutions to this equation, where the coefficient function  $a$ can be either strictly sign-definite or sign-changing. As an application, we derive Liouville theorems for ancient and eternal solutions on manifolds with nonnegative Ricci curvature, generalizing a number of classical results. Our proof features the incorporation and tuning of two parameters in the auxiliary quantities to accommodate the variable coefficient $a$ and to extend the admissible range of $p$.
\end{abstract}
\maketitle

\section{Introduction}
Let $(M,g)$ be an $n$-dimensional complete Riemannian manifold with Ricci curvature satisfying
\begin{equation}\label{e:ric-condition}
    \mathrm{Ric}\geq -(n-1)\kappa g,
\end{equation}
where $\kappa\geq 0$ is a constant. In this paper, we study the semilinear parabolic equation
\begin{equation}\label{e:main}
\partial_t f-\Delta f=a(x, t)f^p,
\end{equation}
where $p\in \mathbb{R}$ and the function $a(x, t)$ is $C^2$ in the first variable and $C^1$ in the second variable. Our goal is to establish a Li-Yau type local gradient estimate and Liouville theorem for positive solutions of \eqref{e:main}, where the coefficient $a(x, t)$ is allowed to change sign.

Let us first recall some gradient estimate results and Liouville theorems related to equation~\eqref{e:main}. When $p=1$, equation~\eqref{e:main} becomes the linear equation
\begin{equation}\label{e:linear}
\partial_t f-\Delta f= a(x,t)f.  
\end{equation}
In the celebrated work \cite{LY86}, Li and Yau established the differential Harnack inequality for positive solutions of \eqref{e:linear}, and studied its fundamental solution in some special situations, including the heat equation, which yields many applications. Important variants of gradient estimates for the heat equation were later obtained by Hamilton~\cite{Hamilton1993} and Souplet-Zhang~\cite{SZ06}. In particular, when $\kappa=0$, Souplet-Zhang's gradient estimate implies a Liouville theorem for positive eternal solutions to the heat equation under certain growth conditions.

When $a\equiv 1$ and $M=\R^n$ is the Euclidean space, the equation
\begin{equation}\label{eq:Fujita}
\partial_t f-\Delta f=f^p
\end{equation}
was first studied by Fujita~\cite{F66}. In particular, Fujita~\cite{F66} proved that for
\[ 1<p<p_F:=1+\frac{2}{n},\]
any positive solution to equation \eqref{eq:Fujita} blows-up in finite time. It follows that equation \eqref{eq:Fujita} admits no positive global solutions on $M\times[0,+\infty)$. The same nonexistence result also holds for the critical exponent $p=p_F$~\cite{H73,KST77}. Since then, the blow-up phenomenon, qualitative behavior and nonexistence of ancient, global or eternal solutions to \eqref{eq:Fujita} on Euclidean spaces has been extensively studied. See, for example, \cite{GK1985, GK1987, GK1989, BV1998, Levine1990, Q16, Q21, QS07, PQS07, WWW2025, MZ2026}.

When $M$ is a Riemannian manifold, Zhang~\cite{Zhang1998, Zhang1999} generalized Fujita's result to the more general case where $a=a(x)$ is variable and positive. In fact, he derived the nonexistence of global positive solutions under certain assumptions both on the manifold and the function $a(x)$. This was further generalized by Castorina-Mantegazza-Sciunzi \cite{CMS19} to equation~\eqref{e:main} where the function $a=a(x,t)>0$ can be time-dependent and satisfies some growth conditions. 

It is well known from a classical result of Gidas-Spruck~\cite{GS1981} that the elliptic equation
\[ -\Delta f = f^p, \]
admits no positive solutions on Riemannian manifolds with nonnegative Ricci curvature for $1 < p < p_S$, where
\[ p_S(n)=+\infty \text{ for } n=1, 2, \quad p_S(n):=\frac{n+2}{n-2}\ \text{ for } n\geq 3.\] 
It is natural to expect the parabolic analog to hold, i.e. equation \eqref{eq:Fujita} admits no positive eternal solution for $1<p<p_S$. On Euclidean spaces, this conjecture was settled by Quittner \cite{Q16, Q21}. However, on general Riemannian manifolds with nonnegative Ricci curvature, the Liouville theorem for all subcritical exponents $1 < p < p_S$ remains unsolved.

One promising way to tackle the above conjecture is to establish Li-Yau or Souplet-Zhang type gradient estimates. There has been much progress along this direction in recent years, see for example \cite{CCK15, CM2021, CCM23, CDM26, LZ25}. In particular, the very recent work of Chen, Du and Ma \cite{CDM26} proved a Li-Yau type gradient estimate for positive solutions to \eqref{eq:Fujita} on a Riemannian manifold with Ricci curvature satisfying (\ref{e:ric-condition}) and $1<p<p_D$, where
\begin{equation}\label{e:p-d}
  p_D=8\,\,\,\text{for}\,\, n=1,\quad p_D=\frac{3+\sqrt{17}}{2}\,\,\,\text{for}\,\, n=2, \quad p_D=\frac{n+2+\sqrt{n^2+8n}}{2(n-1)}\,\,\,\text{for}\,\, n\geq3.
\end{equation}
More precisely, they showed that for any positive eternal solutions $f$ to \eqref{eq:Fujita} with $1<p<p_D$,
there exist positive constants $\beta$, $\gamma$ and $C_0$ depending only on $n$ and $p$, such that
\begin{equation}\label{e:du}
    \frac{|\nabla f|^2}{f^2}-\gamma\frac{f_t}{f} + \beta f^{p-1} \leq C_0\kappa.
\end{equation}
Estimate \eqref{e:du} directly implies a Liouville theorem for positive eternal solutions to equation \eqref{eq:Fujita} when $\kappa=0$. To the best of our knowledge, \eqref{e:p-d} seems to be the largest range of the exponent $p$ for Li-Yau type gradient estimates in the existing literature for equation \eqref{eq:Fujita}.

For the equation~\eqref{e:main} with variable coefficient $a=a(x,t)$, few results on Li-Yau type gradient estimates exist in the literature. In 1991, J. Li~\cite{L1991} established a Li-Yau inequality for $0 <p< \frac{n}{n-1}$ and $n\ge 4$. More precisely, if $a(x,t)$ is nonnegative and satisfies
\[ \p_t a + \Delta a \ge 0,\]
then a positive solution to \eqref{e:main} satisfies
\begin{equation}\label{e:li1991}
    p\frac{|\n f|^2}{f^2} -\frac{f_t}{f} + a(x,t)f^{p-1} \le C(n,p)\left(\frac{1}{t} + \ka\right).
\end{equation} 
Here we also mention some related results in the elliptic case with variable coefficients. In~\cite{GS1981}, Gidas-Spruck actually established a Liouville theorem on Ricci nonnegative manifolds for the equation
\begin{equation}\label{e:GS1981-a}
      \Delta f + a(x)f^p = 0,
\end{equation}
where $1\le p<p_S$ and $a(x)\ge 0$ satisfies $\Delta a\ge 0$ and certain growth conditions. Li~\cite{L1991} showed that for $n\ge 4$ and $1<p<\frac{n}{n-2}$, there is no positive solution to \eqref{e:GS1981-a}, only assuming $a(x)$ satisfies $\Delta a(x)\ge 0$ and $\sup_M a>0$.   

We remark that Cheng-Yau type gradient estimates and Liouville theorems have been obtained for many elliptic equations with more general nonlinear terms with variable coefficients. For example, Wang and Wang ~\cite{WjW24, WjW24*} established interior estimates of the equation
\begin{equation}\label{eq-Wang}
\Delta f+a(x)f\ln f+b(x)f=0,
\end{equation}
where $a(x)$ is either positive or negative. Recently, the authors of the current paper~\cite{SW25} obtained gradient estimates for the equation: 
\begin{equation}\label{e:song-wumain}
\Delta f+a(x)f^p(\ln(f+c))^q=0.
\end{equation}

\

In this paper, we derive Li-Yau type gradient estimates and Liouville theorems for the equation \eqref{e:main}, where $a(x,t)$ is variable and the underlying manifold satisfies Ricci lower bound \eqref{e:ric-condition}. 

Let $P_{R,T}=B(o,R)\times (t_0-T,t_0]\subset M\times \R$ be a parabolic cylinder of radius $R>0$ and $T>0$. Suppose $a(x,t)$ is  not identically zero  on $P_{R,T}$, which is $C^2$ in the first variable and $C^1$ in the second variable. Our main results can be divided into two cases: $a(x,t)$ is strictly sign-definite or allowed to change sign. 

Our first result extends the gradient estimate of Chen-Du-Ma~\cite{CDM26} to \eqref{e:main} with variable coefficients that are either strictly positive or strictly negative.

\begin{thm} \label{thm:main1}
Suppose $(M,g)$ is an n-dimensional complete Riemannian manifold satisfying \eqref{e:ric-condition} and $n\geq 3$.  
Let $p_D$ be defined by \eqref{e:p-d}. Suppose that $p$ satisfies
\begin{itemize}
\item [(P1)] $p<p_D$, if $a(x, t)$ is positive on $P_{R,T}$;  
\item [(P2)] $p>1$, if $a(x, t)$ is negative on $P_{R,T}$,
\end{itemize}
and $a$ satisfies
\begin{itemize}
\item[(A1)] There exists $K\geq0$ such that
\begin{equation}\label{assumption-a-1}
\left|\frac{\p_t a}{a}\right| + \left|\frac{\Delta a}{a}\right|+\left|\frac{\nabla a}{a}\right|^2\leq K,
\text{ on } P_{R,T}.
\end{equation}
\end{itemize}
Then there exist constants $\lambda\in(0,1)$ and $\de_0>0$, only depending on $n$ and $p$, such that any positive solution $f$ to equation \eqref{e:main} on $P_{R,T}$ satisfies
\begin{equation}\label{ineq:zhuyao1}
\lambda\frac{|\nabla f|^2}{f^2}-\frac{f_t}{f}+\de_0 |a(x,t)|f^{p-1}\leq C(n,p,\lambda,\de_0)\left(\frac{1}{T}+\frac{1}{R^2}+\ka+K\right)
\end{equation}
on $P_{R/2, T/2}$, where $C(n, p,\lambda,\de_0)$ is a positive constant depending on $n$, $p$, $\lambda$ and $\de_0$.
\end{thm}

\begin{rmk}\label{rmk:1}
    Under assumption $(A1)$, since $|\n a|\leq \sqrt{K}|a|$, the function $a$ in fact satisfies a Harnack inequality, which implies that $a$ must be strictly sign-definite.
\end{rmk}

When $a$ is allowed to change sign (but not necessarily sign-changing), we have the following Li-Yau type gradient estimate for $n\ge 4$ which extends the result of Li~\cite{L1991} to $1<p<\frac{n}{n-2}$.

To state our results, we define the operator:
\[
    L_\lambda :=\De+(\frac{p}{\lambda}-1)\partial_t,
\]
where $\lambda$ is to be determined.
\begin{thm} \label{thm:main2}
Suppose $(M,g)$ is an n-dimensional complete Riemannian manifold satisfying \eqref{e:ric-condition} and $n\ge 4$.
Suppose $p$ satisfies
\begin{itemize} 
\item[(P3)] $1<p<\frac{n}{n-2}$,
\end{itemize}
and $a$ satisfies either of the following assumptions:
\begin{itemize}
\item[(A2)] $L_{\frac{p}{2}}a=( \Delta+\p_t)a\ge 0$ on $P_{R,T}$;
\item[(A3)] there exists $K\geq 0$ such that $\left|(\p_t + \Delta)a\right|\leq K|a|$ on 
$P_{R, T}$.
\end{itemize}
Then any positive solution $f$ to equation \eqref{e:main} on $P_{R,T}$ satisfies
\begin{equation}\label{ineq:zhuyao2}
\frac{p}{2}\frac{|\nabla f|^2}{f^2}-\frac{f_t}{f}+\frac{1}{2} a(x,t)f^{p-1}\leq C(n,p)\left(\frac{1}{T}+\frac{1}{R^2}+\ka+K\right)
\end{equation}
on $P_{R/2, T/2}$, where $K=0$ in case $(A2)$, and $C(n, p)$ is a positive constant depending on $n$ and $p$.
\end{thm}

Moreover, for $n\le 3$ and $1<p<\frac{8}{n}$ (which is larger than $p_D$ in \eqref{e:p-d} for $n=2$), we have

\begin{thm} \label{thm:main3}
Suppose $(M,g)$ is an n-dimensional complete Riemannian manifold satisfying \eqref{e:ric-condition} and $n\le 3$. For any $p$ satisfying
\begin{itemize}
\item[(P4)] $1<p<\frac{8}{n}$,
\end{itemize}
there exists a constant $\ld\in(0,1)$ depending only on $n$ and $p$, such that if $a$ satisfies either of the following assumptions:
\begin{itemize}
\item[(A4)] $L_{\lambda}a=(\De+(\frac{p}{\lambda}-1)\partial_t)a\ge 0$ on $P_{R,T}$;
\item[(A5)] there exists $K\geq 0$ such that $\left|L_{\lambda}a\right|\leq K|a|$ on $P_{R, T}$,
\end{itemize}
then any positive solution to equation \eqref{e:main} on $P_{R,T}$ satisfies
\begin{equation}\label{ineq:zhuyao3}
\ld\frac{|\nabla f|^2}{f^2}-\frac{f_t}{f}+\frac{\ld}{p} a(x,t)f^{p-1}\leq C(n,p,\ld)\left(\frac{1}{T}+\frac{1}{R^2}+\ka+K\right)
\end{equation}
on $P_{R/2, T/2}$, where $K=0$ in case $(A4)$, and $C(n, p,\ld)$ is a positive constant depending on $n$, $p$ and $\ld$.
\end{thm}

\begin{rmk}
    The difference between Theorem~\ref{thm:main2} and Theorem~\ref{thm:main3} lies in the fact that, when $n\le 3$, the choice of the constant $\ld=\frac{p}{2}$ in Theorem~\ref{thm:main2} might exceed the admissible range $(0,1)$. In fact, in Theorem~\ref{thm:main3}, the constant $\ld$ is taken to be sufficiently close to 1.
\end{rmk}

\begin{rmk}
    Actually, we only need to assume that the function $a$ satisfies either $(A2)$ or $(A3)$ at each point $(x,t)\in P_{R,T}$ in Theorem~\ref{thm:main2}, and the result still holds true. Similarly, we may assume that $a$ satisfies either $(A4)$ or $(A5)$ at each point $(x,t)\in P_{R,T}$ in Theorem~\ref{thm:main3}.
\end{rmk}

\begin{rmk}
There are plenty of functions satisfying the assumptions on $a$ in Theorem~\ref{thm:main2} and Theorem~\ref{thm:main3}. For example, the following functions satisfy Theorem~\ref{thm:main3}.
\begin{itemize}
    \item $a(x,t) = a(x)$ only depends on $x$, and
    \[ \Delta a \ge 0 \text{ or } |\Delta a| \le K|a|; \]
    \item $a(x, t) = a(t)$ only depends on $t$, and 
    \[ \p_t a\ge 0 \text{ or } (\frac p\ld-1)|\p_t a|\le K|a|;\]
    \item Let $\phi$ be a Laplace eigenfunction of $M$ satisfying $\Delta \phi = \eta \phi$ for some $\eta\ge 0$, then the function
    \[ a(x,t) = e^{ct}\phi(x) \]
    satisfies assumption (A5) with
    \[ K = |(\frac p\ld-1)c +  \eta|.\]
\end{itemize}
\end{rmk}

As an application of the above estimates, we get the following Liouville theorems for equation~\eqref{e:main}, which extend previous results. When $a\equiv c$ is a nonzero constant, the following Liouville theorem for the equation
\begin{equation}\label{e:c}
    \p_t f - \Delta f = c f^p
\end{equation} 
follows from Theorem~\ref{thm:main1}. 
In particular, it also applies to corresponding elliptic equations~\eqref{e:GS1981-a} with constant coefficient, which was extensively studied (cf. \cite{GS1981, WW23, HWW24}).

\begin{cor}\label{thm:liouville1}
Suppose $M$ is an $n$-dimensional $(n\geq 3)$ complete non-compact Riemannian manifold with nonnegative Ricci curvature. 
\begin{itemize}
    \item[(1)] If $$c>0, \quad 1<p<p_D$$
    or $$c<0, \quad p> 1,$$ then equation \eqref{e:c} admits no positive eternal solutions on $ M\times(-\infty,+\infty)$.
    \item[(2)] If $$c>0, \quad p<1,$$ then equation \eqref{e:c} admits no positive ancient solutions.
\end{itemize}
\end{cor}

\begin{rmk}
    When $p=1$ and the linear equation
    \[ \p_t f - \Delta f = c f,\]
    does admit non-trivial positive eternal solutions. For example, even on the Euclidean space, there is a non-trivial solution $f = e^{ct}$.
\end{rmk}

For a nonconstant function $a$, the following Liouville theorem for ancient or eternal solutions to \eqref{e:main} follows directly from Theorem~\ref{thm:main2}. To our knowledge, this seems to be the first Liouville theorem for equation~\eqref{e:main} which allows $a(x,t)$ to change sign.

\begin{cor} \label{thm:liouville2}
    Suppose $M$ is an $n$-dimensional complete non-compact Riemannian manifold with nonnegative Ricci curvature and $n\ge 4$. Let $T_0$ be a real number or $+\infty$. Suppose $1<p<\frac{n}{n-2}$ and $a(x,t)$ satisfies
\begin{itemize}
    \item $(\Delta+\p_t)a\ge 0$ on $M\times(-\infty, T_0)$,
    \item there exists a point $x_0\in M$ and $T_1\in(-\infty, T_0)$ such that 
    \[
    \lim_{s\rightarrow T_0}\int_{T_1}^{s}a(x_0,t)\,dt=+\infty.
\]
\end{itemize}  
Then equation~\eqref{e:main} admits no positive solution on $M\times(-\infty,T_0)$.
\end{cor}

A similar Liouville theorem for $n\le 3$ follows from Theorem~\ref{thm:main3}, which we omit. In particular, when the coefficient function $a=a(x)$ is independent of $t$, we also get a Liouville theorem for eternal solutions to the equation
\begin{equation}\label{e:a-x}
    \p_t f - \Delta f = a(x) f^p.
\end{equation}
The result also applies to the corresponding elliptic equation~\eqref{e:GS1981-a} with variable coefficient, which extends the classical results of Gidas-Spruck~\cite{GS1981} and Li~\cite{L1991}.

\begin{cor} \label{thm:liouville3}
    Suppose $M$ is an $n$-dimensional complete non-compact Riemannian manifold with nonnegative Ricci curvature. Suppose 
\begin{itemize}
    \item $1<p<\frac{n}{n-2}$ for $n\ge 4$,
    \item $1<p<\frac{8}{n}$ for $n\le 3$,
\end{itemize}
and
\begin{itemize}
    \item $\Delta a(x)\ge 0$ on $M$,
    \item $\sup_{x\in M} a(x) >0$.
\end{itemize}
Then equation~\eqref{e:a-x} admits no positive eternal solutions on $M\times (-\infty, +\infty)$.
\end{cor}

\

Our proofs of main results rely on a differential inequality of a Li-Yau type auxiliary quantity and Nash-Moser iteration. A key feature of our gradient estimates, as compared to previous studies, lies in the careful choice of the parameters $\lambda$ and $\delta$ in \eqref{ineq:zhuyao1}. In fact, to obtain the optimal range of $p$ that guarantees the validity of such gradient estimates, one needs to fully exploit the adjustability of these parameters. This idea of modulating parameters is partially inspired by \cite{PWW21, HW22}, where Cheng-Yau type gradient estimates were derived for a class of semilinear elliptic equations, and has been explored in~\cite{L1991, CDM26}.

Compared to the result \eqref{e:du} of Chen-Du-Ma~\cite{CDM26}, the presence of the variable coefficient $a(x, t)$ in \eqref{e:main} introduces additional analytical challenges. Specifically, in deriving the differential inequalities for a Li--Yau-type quantity, the derivatives of $a$ yield lower-order terms involving $\nabla a$, $\Delta a$, and $\p_t a$. These terms can be effectively controlled when they are bounded proportionally to $a$ itself. This observation motivates assumption (A1) in Theorem~\ref{thm:main1}. However, as mentioned in Remark~\ref{rmk:1}, $a$ must be strictly sign-definite under assumption (A1). 

To derive the Li--Yau-type estimate in Theorem~\ref{thm:main2} and Theorem~\ref{thm:main3}, where $a$ is allowed to change sign, new observations are needed. We find that under a specific choice of the parameters (i.e. $\ld = p\delta$) within the auxiliary quantity in \eqref{ineq:zhuyao2} and \eqref{ineq:zhuyao3}, the terms containing $\nabla a$ vanish, leaving only the terms involving $\Delta a$ and $\partial_t a$, which can be handled under assumptions~(A2-A5).

\

The rest of the paper is organized as follows. In Section~\ref{s:pointwise}, we define a Li-Yau type quantity with modulating parameters and derive a key differential inequality. Next in Section~\ref{s:integral}, using the Saloff-Coste’s Sobolev inequality, we derive an integral inequality and an initial integral estimate where the Nash-Moser iteration is readily applied. Finally, Section~\ref{s:proof} contains the proofs of our main theorems and corollaries.

\section{Pointwise differential inequality}\label{s:pointwise}

Suppose $(M,g)$ is an $n$-dimensional complete Riemannian manifold with Ricci curvature lower bound \eqref{e:ric-condition}. Suppose $f$ is a classical positive solution to equation \eqref{e:main} on $P_{R,T}$. As in the existing literature, we make use of the logarithmic transformation
$$g:=\ln f.$$
A direct computation shows
\begin{equation}\label{ht2-eq}
\p_t g-\De g=|\nabla g|^2+a(x,t)e^{(p-1)g}.
\end{equation}

Set
\[ G: = \lambda\frac{|\nabla f|^2}{f^2} -  \frac{\p_t f}{f}+\de a(x,t)f^{p-1}=\lambda|\nabla g|^2- g_t+\de a(x,t)e^{(p-1)g}\]
where $\ld\in (0,1)$ and $\de$ are constants to be determined later. 

We first derive a differential identity for the quantity $G$. We start with
\begin{equation} \label{eq:operator_G}
(\partial_t-\De)G = \lambda(\partial_t-\De)|\nabla g|^2 - ( \partial_t-\Delta  )g_t + \delta(\partial_t-\De)\big(a(x,t)e^{(p-1)g}\big).
\end{equation}
For the first term on the right-hand side of \eqref{eq:operator_G}, by the Bochner formula and the identity
$$\lambda|\n g|^2=g_t+G-\de a(x,t)e^{(p-1)g},$$ 
we have
\begin{align*}
\lambda( \partial_t-\Delta )|\nabla g|^2 &= -2\lambda|\nabla^2 g|^2 + 2\lambda\langle \nabla g, \nabla( g_t-\Delta g ) \rangle - 2\lambda\mathrm{Ric}(\n g,\n g)\\
= &-2\lambda|\nabla^2 g|^2 +2\lambda\langle \nabla g, \nabla|\nabla g|^2 \rangle + 2\lambda\langle \nabla g, \nabla\big(a(x,t)e^{(p-1)g}\big) \rangle- 2\lambda\mathrm{Ric}(\n g,\n g)\\
=&-2\lambda|\nabla^2 g|^2 +2\langle \nabla g, \nabla g_t \rangle+2\langle\n g,\n G\rangle+2(\lambda-\de) e^{(p-1)g}\langle\n g,\n a\rangle \\
&+2(\lambda-\de)(p-1) a(x,t)e^{(p-1)g}|\n g|^2- 2\lambda\mathrm{Ric}(\n g,\n g).
\end{align*}
For the second term in \eqref{eq:operator_G}, taking the derivative of both sides of the equation $\eqref{ht2-eq}$ with respect to $t$ and using $$g_t=\lambda|\n g|^2-G+\de a(x,t)e^{(p-1)g},$$we get
\begin{align*}
(\partial_t-\Delta  )g_t=&2\langle\n g,\n g_t\rangle+(p-1)a(x,t)e^{(p-1)g}g_t+\partial_t ae^{(p-1)g}\\
=&2\langle\n g,\n g_t\rangle+\lambda(p-1)a(x,t)e^{(p-1)g}|\n g|^2-(p-1)a(x,t)e^{(p-1)g}G\\
&+\de(p-1)a^2(x,t)e^{2(p-1)g}+\partial_t ae^{(p-1)g}.
\end{align*}
For the third term in \eqref{eq:operator_G}, we have
\begin{align*}
\de(\partial_t-\De)(a(x,t)e^{(p-1)g})&= \de (p-1)a(x,t)e^{(p-1)g}g_t+\de e^{(p-1)g}(\partial_t-\De) a\\
&-2\de (p-1)e^{(p-1)g}\langle\n g,\n a\rangle-\de (p-1)^2a(x,t)e^{(p-1)g}|\n g|^2\\
&-\de (p-1)a(x,t)e^{(p-1)g}\De g\\
&=\de(2-p)(p-1)a(x,t)e^{(p-1)g}|\n g|^2+\de (p-1)a^2(x,t)e^{2(p-1)g}\\
&+\de e^{(p-1)g}(\partial_t-\De) a-2\de (p-1)e^{(p-1)g}\langle\n g,\n a\rangle.
\end{align*}
Consolidating the above identities into \eqref{eq:operator_G}, we obtain
\begin{equation}
    \label{eq:Gmian2}
    \begin{aligned}
(\partial_t-\Delta )G =  & -2\lambda|\nabla^2 g|^2 - 2\lambda\mathrm{Ric}(\n g,\n g) + 2\langle\nabla g, \nabla G\rangle \\
&+ (p - 1)(\lambda - \delta p)  a(x,t)e^{(p-1)g} |\nabla g|^2  + (p - 1)a(x,t)e^{(p-1)g} G\\ &+2(\lambda - \delta p)e^{(p-1)g} \langle \nabla g, \nabla a \rangle  -\de e^{(p-1)g}\De a+(\de-1)e^{(p-1)g}\partial_ta.
\end{aligned}
\end{equation}

Next, we estimate the Hessian term by the elementary algebraic inequality
\[ |\nabla^2 g|^2 \ge \frac{1}{n}(\Delta g)^2. \]
Using the equation \eqref{ht2-eq}, we may express the Laplacian of $g$ in terms of $G$ by
\[
\Delta g = -G - ( 1-\lambda)|\nabla g|^2 - (1-\delta)a(x,t)e^{(p-1)g}.
\]
It follows
\begin{equation}
    \label{ineq:hessg}
    \begin{aligned}
    |\nabla^2 g|^2 \ge & \frac1{n} \Big[ G^2 + (1-\lambda )^2|\nabla g|^4 + (1-\delta)^2 a^2(x,t)e^{2(p-1)g} \\
& + 2( 1-\lambda)|\nabla g|^2 G + 2(1-\delta)a(x,t)e^{(p-1)g}G \\
& + 2( 1-\lambda)(1-\delta)a(x,t)e^{(p-1)g}|\nabla g|^2 \Big].
\end{aligned}
\end{equation}
Keeping the Ricci curvature lower bound \eqref{e:ric-condition} in mind and substituting $\eqref{ineq:hessg}$ into $\eqref{eq:Gmian2}$, we arrive at
\begin{equation}\label{e:G-ineq}
    \begin{aligned}
(\partial_t-\Delta )G \le  & -\frac{2\lambda}{n} G^2 - \frac{4\lambda}{n}(1-\lambda)|\nabla g|^2 G + \left[ (p - 1) - \frac{4\lambda}{n}(1 - \delta) \right] a(x,t)e^{(p-1)g} G \\
& + 2\langle\nabla g, \nabla G\rangle - \frac{2\lambda}{n}(1-\lambda )^2 |\nabla g|^4 \\
& + \left[  (p - 1)(\lambda - \delta p)-\frac{4\lambda}{n}(1-\lambda)(1 - \delta)  \right] a(x,t)e^{(p-1)g} |\nabla g|^2 \\
& + 2\lambda(n - 1)\kappa |\nabla g|^2 +2(\lambda - \delta p)e^{(p-1)g} \langle \nabla g, \nabla a \rangle \\
& - \frac{2\lambda}{n}(1 - \delta)^2 a^2(x,t) e^{2(p-1)g}-\de e^{(p-1)g}\De a+(\de-1)e^{(p-1)g}\partial_ta.
\end{aligned}
\end{equation}

To proceed, we define the quadratic function
\begin{equation}\label{e:H}
    \begin{aligned}
H(X, Y, Z):= &-\frac{2\lambda}{n} X^2- \frac{2\lambda}{n}(1-\lambda )^2 Y^2 -\frac{4\lambda}{n}(1-\lambda)XY- \frac{2\lambda}{n}(1- \delta)^2 Z^2 \\
&+ \left[(p-1)-\frac{4\lambda}{n}(1 -\delta)\right] XZ + \left[(p - 1)(\lambda - \delta p)-\frac{4\lambda}{n}(1-\lambda)(1 - \delta)\right] YZ.
\end{aligned}
\end{equation}
where $X> 0, Y\ge 0$ and $Z\in \R$. We need the following two technical lemmas.

\begin{lem}\label{lem:H_estimate}
Suppose $n\geq 3$, $p$ and $Z$ satisfy one of the following assumptions:
\begin{itemize}
    \item [(Z1)] $Z\geq 0$ and $p<p_D:=\frac{n+2+\sqrt{n^2+8n}}{2(n-1)}$;
    \item [(Z2)] $Z\leq 0$ and $p>1$;
    
\end{itemize}
Then there exist constants $\lambda\in(0,1)$, $\de \in \mathbb{R}$, and $\ep_1, \ep_2, \ep_3>0$ such that
\begin{equation}\label{e:lem1}
    H(X,Y,Z) \leq -\ep_1X^2-\ep_2XY-\ep_3Z^2.
\end{equation}
Here these constants depend only on $n$ and $p$. Moreover, we can choose $\de>0$ in case (Z1) and $\de<0$ in (Z2), respectively.
\end{lem}

\begin{proof}
We prove this lemma case by case.

\noindent
\textbf{Case (Z1):} $Z\geq 0$ and $p<p_D$.

We first rewrite \eqref{e:H} as
\begin{equation*}
H(X,Y,Z)=-\frac{2\lambda}{n}\big(X+(1-\lambda)Y+(1-\de)Z\big)^2+(p-1)XZ+(p-1)(\lambda-\de p)YZ.
\end{equation*}

If $p\leq 1$, since $X> 0$ and $Z\geq 0$, we have $(p-1)XZ\leq 0$. By choosing $\de \in (0, 1)$ sufficiently small such that $\lambda-\de p>0$, we also have $(p-1)(\lambda-\de p)YZ \le 0$. It follows that
\begin{equation}\label{con-3}
\begin{aligned}
H(X,Y,Z) &\leq -\frac{2\lambda}{n}\big(X+(1-\lambda)Y+(1-\de)Z\big)^2 \\
&\leq -\frac{2\lambda}{n} X^2 - \frac{4\lambda}{n}(1-\lambda)XY - \frac{2\lambda}{n}(1 - \delta)^2 Z^2.
\end{aligned}
\end{equation}

If $p>1$, we introduce a new variable $W:=X+(1-\lambda)Y>0$. Substituting $X=W-(1-\lambda)Y$ into $H$, we get
\begin{equation}\label{eq:H}
\begin{aligned}
H(W,Y,Z)&=-\frac{2\lambda}{n}(W+(1-\de)Z)^2+(p-1)\big(W-(1-\lambda)Y\big)Z+(p-1)(\lambda-\de p)YZ\\
&=-\frac{2\lambda}{n}(W+(1-\de)Z)^2+(p-1)WZ+(p-1)(2\lambda-\de p-1)YZ.
\end{aligned}
\end{equation}
Now set $\de = \frac{2\lambda-1}{p}$, so that the coefficient of the $YZ$ term vanishes. For a sufficiently small constant $s>0$, we rewrite and estimate $H$ as follows:
\begin{equation}\label{ineq:H}
\begin{aligned}
H(W,Y,Z) \leq & -\left(\frac{2\lambda}{n}-s\right)(W+(1-\de)Z)^2+(p-1)WZ-s(W+(1-\de)Z)^2\\
= & -W^2\bigg\{\left(\frac{2\lambda}{n}-s\right)-\left(p-1-\left(\frac{4\lambda}{n}-2s\right)(1-\de)\right)\frac{Z}{W} \\&+\left(\frac{2\lambda}{n}-s\right)(1-\de)^2\left(\frac{Z}{W}\right)^2\bigg\} -s(W+(1-\de)Z)^2\\
\leq & W^2\frac{(p-1)}{(1-\de)^2}\left(\frac{n(p-1)}{8\lambda-4ns}-(1-\de)\right)-s(W+(1-\de)Z)^2.
\end{aligned}
\end{equation}
Here, we use the basic algebraic inequality $a - bx + cx^2 \geq a - \frac{b^2}{4c}$ in the last line. 
Next, we choose $\lambda=\frac{p+1}{4}$. Note that $\ld\in (\frac12, 1)$ since $1<p<p_D<3$.
With these specific choices of $\de$ and $\ld$, we calculate
\begin{align*}
\frac{n(p-1)}{8\lambda}-(1-\de)&=\frac{n(p-1)}{2(p+1)}-1+\frac{p-1}{2p}\\
&=\frac{1}{2p(p+1)}\left[(n-1)p^2-(n+2)p-1\right].
\end{align*}
Observe that $(n-1)p^2-(n+2)p-1<0$ precisely when $1<p<p_D$. Therefore, for such $p$, we can choose $s_1>0$ small enough (depending on $n, p, \lambda$, and $\de$) such that
\[
\frac{n(p-1)}{8\lambda-4ns_1}-(1-\de)\leq 0.
\]
Consequently, the $W^2$ term in \eqref{ineq:H} is non-positive, yielding
\begin{equation}\label{con-4}
\begin{aligned}
H(X,Y,Z)&\leq -s_1(W+(1-\de)Z)^2= -s_1\left(X+\frac{3-p}{4}Y+\frac{p+1}{2p}Z\right)^2\\
&\leq -s_1X^2-s_1\frac{3-p}{2}XY-s_1\left(\frac{p+1}{2p}\right)^2Z^2.
\end{aligned}
\end{equation}

From \eqref{con-3} and \eqref{con-4}, we see that \eqref{e:lem1} holds in Case (Z1).

\noindent
\textbf{Case (Z2):} $Z\leq 0$ and $p>1$.

If $p\geq 1+\frac{4}{n}$, for any $\lambda\in (0,1)$, both coefficients of the last two terms in \eqref{e:H} are strictly positive at $\de=0$. Hence, by continuity, we choose $\de<0$ with $|\de|$ sufficiently small such that
\[
 (p - 1) - \frac{4\lambda}{n}(1 - \delta) \geq 0 \quad \text{and} \quad (p - 1)(\lambda - \delta p)-\frac{4\lambda}{n}(1-\lambda)(1 - \delta) \geq 0.
\]
Since $Z \leq 0$ and by assumption $X, Y \ge 0$, we can simply dorp the non-positive terms involving $XZ$, $YZ$ and $Y^2$, yielding
\begin{equation}\label{con-1}
H(X,Y,Z)\leq -\frac{2\lambda}{n} X^2 - \frac{4\lambda}{n}(1-\lambda)XY- \frac{2\lambda}{n}(1 - \delta)^2 Z^2.
\end{equation}

If $1<p<1+\frac{4}{n}$, there exists a sufficiently small $\lambda\in(0,1)$ such that 
\[ 1+\frac{4\lambda}{n} < p < 1+\frac{4(1-\lambda)}{n}.\]
Then we choose $\de<0$ with $|\de|$ sufficiently small such that
\begin{equation}\label{e:delta-1}
    p -1- \frac{4\lambda}{n} (1-\de) > 0
\end{equation}
and
\begin{equation}\label{e:delta-2}
    -\frac{4\lambda(1-\lambda)}{n}(1-\de)  < (p - 1)(\lambda-\de p)-\frac{4\lambda}{n}(1-\lambda)(1-\de) < 0.
\end{equation}

Now \eqref{e:delta-1} ensures the $XZ$ term in \eqref{e:H} is non-positive. For the $YZ$ term, we apply Young's inequality to get
\begin{align*}
&\left[ (p - 1)(\lambda - \delta p)-\frac{4\lambda}{n}(1-\lambda)(1 - \delta) \right] YZ\\
&\quad \leq \frac{2\lambda(\lambda-1)^2}{n}Y^2 +\frac{n}{8\lambda(\lambda-1)^2}\left[(p - 1)(\lambda - \delta p)-\frac{4\lambda}{n}(1-\lambda)(1 - \delta) \right]^2 Z^2.
\end{align*}
Inserting into \eqref{e:H} and dropping the non-positive terms, we arrive at
\begin{equation}\label{con-2}
\begin{aligned}
H(X,Y,Z)\leq & -\frac{2\lambda}{n} X^2 - \frac{4\lambda}{n}(1-\lambda)XY\\ 
&-\left[\frac{2\lambda}{n}(1-\de)^2-\frac{n}{8\lambda(\lambda-1)^2}\left((p - 1)(\lambda - \delta p)-\frac{4\lambda}{n}(1-\lambda)(1 - \delta) \right)^2\right]Z^2,
\end{aligned}
\end{equation}
By \eqref{e:delta-2}, the quantity in square brackets is strictly positive. Consequently, the coefficient of $Z^2$ is strictly negative. 

By \eqref{con-1} and \eqref{con-2}, we see that \eqref{e:lem1} holds in Case (Z2). This finishes the proof of the lemma.
\end{proof}

The following lemma gives the desired bound of $H$ when $Z$ is allowed to change sign.

\begin{lem}\label{lem:H_estimate2}
Suppose that $p$ satisfies the following condition:
\begin{itemize}
    \item [(Z3)] $1<p<\frac{n}{n-2}$ when $n>3$ or $1<p<\frac{8}{n}$ when $n\leq 3$.    
\end{itemize}
Then there exist constants $\lambda\in(0,1)$, $\de=\frac{\lambda}{p}<1 $, $\ep_3>0$ such that the following holds. For every $|\theta|<\ep_3$, there exist $\ep_1(\theta)>0, \ep_2(\theta)>0$ such that
\begin{equation}\label{e:lem2}
    H(X,Y,Z) \leq -\ep_1(\theta)X^2-\ep_2(\theta)XY-\theta XZ.
\end{equation}
In particular, we may choose $\lambda=\frac{p}{2}$ if $n\geq 4$ and $\ld = \frac{np+8}{16}$ if $n\leq 3$.
\end{lem}
\begin{proof}

By setting $\lambda - \delta p = 0$, we can write $H(X,Y,Z)$ as
\begin{equation*}
H(X,Y,Z) = -\alpha U^2 + (p-1)XZ,
\end{equation*}
where $\alpha = \frac{2\lambda}{n}$ and $U = X + (1-\lambda)Y + (1-\delta)Z$. Next we introduce a parameter $\theta$ such that
\begin{equation*}
(p-1)XZ = (p-1+\theta)XZ - \theta XZ.
\end{equation*}
By expressing $Z$ as $Z = \frac{1}{1-\delta} \big(U - X - (1-\lambda)Y \big)$, we have
\begin{align*}
(p-1+\theta)XZ &= \frac{p-1+\theta}{1-\delta} X U - \frac{p-1+\theta}{1-\delta} X^2 - \frac{(p-1+\theta)(1-\lambda)}{1-\delta} XY.
\end{align*}
A simple algebraic inequality yields
\begin{equation*}
-\alpha U^2 + \frac{p-1+\theta}{1-\delta} X U \le \frac{(p-1+\theta)^2}{4\alpha (1-\delta)^2} X^2.
\end{equation*}
Substituting this inequality back into the equation for $H$, the variable $U$ is eliminated, leaving
\begin{align*}
H(X,Y,Z) &\le -\theta XZ + \frac{(p-1+\theta)^2}{4\alpha (1-\delta)^2} X^2 - \frac{p-1+\theta}{1-\delta} X^2 - \frac{(p-1+\theta)(1-\lambda)}{1-\delta} XY \\
&= -\theta XZ + X \left[ \left( \frac{(p-1+\theta)^2}{4\alpha (1-\delta)^2} - \frac{p-1+\theta}{1-\delta} \right) X - \frac{(p-1+\theta)(1-\lambda)}{1-\delta} Y \right].
\end{align*}
Recalling that $\al = \frac{2\lambda}{n}$, we can factor out $\frac{n}{8\lambda}$ from the bracketed terms to obtain:
\begin{align*}
H(X,Y,Z) \le -\theta XZ + \frac{n}{8\lambda} X \Big( C_{\theta} X - D_{\theta} Y \Big).
\end{align*}
where the coefficients are given by
\[
    C_\theta=\frac{(p-1+\theta)^2}{(1-\delta)^2} - \frac{8\lambda(p-1+\theta)}{n(1-\delta)}=\frac{p-1+\theta}{np(1-\de)^2}\left(np(p-1+\theta)-8\lambda(p-\lambda)\right),
\]
and
\[
    D_\theta= \frac{8\lambda(1-\lambda)(p-1+\theta)}{n(1-\delta)}.
\]

Now we consider two cases:\\
\textbf{Case (1):} $n\leq 3$ and $1<p<\frac{8}{n}$.\\
In this case, we can choose $\lambda$ sufficiently close to $1$ such that $1<p<\frac{8\lambda}{n}$. For example, we may set $\ld = \frac{np+8}{16}$. Then for every $1-p<\theta<\frac{(p-1)(8\lambda-np)}{np}$, we have
\[
np(p-1+\theta)-8\lambda(p-\lambda)<(np-8\lambda)(p-1)+np\theta<0.
\]
Hence $C_\theta<0$ and $D_\theta>0$.\\
\textbf{Case (2):} $n\geq 4$ and $1<p<1+\frac{2}{n-2}$.\\
In this case, we set $\lambda=\frac{p}{2}<1$. Then for every $1-p<\theta<1-\frac{n-2}{n}p$, we have
\[
np(p-1+\theta)-8\lambda(p-\lambda)=np(p-1+\theta)-2p^2=p((n-2)p-n)+np\theta< 0
\]
Therefore, we again get $C_\theta<0$ and $D_\theta >0$.

To summarize, under the assumption (Z3), we can choose $\delta=\frac{\ld}{p}$ and
\[ \ld = \frac{np+8}{16}, \quad \ep_3:=\frac{1}{2}\min\{p-1,\frac{(p-1)(8\ld-np)}{np}\}, \text{ for } n\leq 3,\]
or
\[ \ld=\frac{p}{2}, \quad  \ep_3:= \frac{1}{2}\min\{p-1,\frac{n-(n-2)p}{n}\}, \text{ for }n\geq 4. \]
Then for every $|\theta|<\ep_3$, there exist constants $\ep_1(\theta):=-\frac{n}{8\lambda}C_{\theta}>0$ and $\ep_2(\theta):=\frac{n}{8\lambda}D_{\theta}>0$, such that \eqref{e:lem2} holds. This completes the proof of the lemma.
\end{proof}

Now we are ready to prove our key differential inequality under the assumptions of Theorem~\ref{thm:main1}, Theorem~\ref{thm:main2} or Theorem~\ref{thm:main3}.

\begin{thm}\label{l:pointwise-ineq1}
Suppose $(M,g)$ is an $n$-dimensional complete Riemannian manifold with Ricci curvature lower bound \eqref{e:ric-condition}. Suppose $n$, $p$ and $a(x,t)$ fulfill one of the following assumptions 
\begin{itemize}
    \item[(T1)] $n\ge 3$, $p$ satisfies $(P1)$ or $(P2)$, and $a(x,t)$ satisfies $(A1)$;
    \item[(T2)] $n\ge 4$, $p$ satisfies $(P3)$ and $a(x,t)$ satisfies $(A2)$ or $(A3)$ at every point of $P_{R,T}$;
    \item[(T3)] $n\le 3$, $p$ satisfies $(P4)$ and $a(x,t)$ satisfies $(A4)$ or $(A5)$ at every point of $P_{R,T}$.
\end{itemize}
In each case, there exist constants $\lambda\in(0,1)$, $\de\in \R$, and $\ep_1, \ep_2, \ep_3>0$, such that the function $u$ defined by 
\begin{equation}\label{e:def-u}
    u := 
    \begin{cases}
        G-\frac{2\lambda(n-1)}{\ep_2}\kappa-\left(\frac{2(\lambda-\delta p)^2}{\ep_2\ep_3}+\sqrt{\frac{\de^2+(1-\de)^2}{\ep_1\ep_3}}\right)K, \quad \text{in Case (T1)}\\
        G-\frac{2\lambda(n-1)}{\ep_2}\kappa-\frac{2\lambda}{p \ep_3}K, \quad \text{in Case (T2) or (T3)}. \\
    \end{cases}
\end{equation}
satisfies the following: at every point $(x,t)\in P_{R,T}$ with $u(x,t)>0$, there holds
\begin{equation}\label{ineq-pt}
(\partial_t-\De)u\leq -\ep_1u^2 - \ep_2|\nabla g|^2 u+2\langle\n g,\n u\rangle.
\end{equation}
In Case (T1), the parameters $\lambda$ and $\delta$ depend only on $n$, $p$ and the sign of $a(x,t)$. In Case (T2) and (T3), they depend only on $n$ and $p$.
\end{thm}

\begin{proof}
Recall that we have already shown that the function $G$ satisfies \eqref{e:G-ineq}. By setting
\[
X=G,\quad Y=|\n g|^2\quad\text{and}\quad Z=a(x,t)e^{(p-1)g},
\]
we can rewrite \eqref{e:G-ineq} as
\begin{equation}
    \label{ineq-G}
    \begin{aligned}
(\partial_t-\Delta )G \le  & H(X, Y, Z) + 2\lambda(n - 1)\kappa Y +2(\lambda - \delta p)e^{(p-1)g} \langle \nabla g, \nabla a \rangle \\
& + 2\langle\nabla g, \nabla X\rangle -\de e^{(p-1)g}\De a+(\de-1)e^{(p-1)g}\partial_t a,
\end{aligned}
\end{equation}
where $H$ is the function defined in \eqref{e:H}. Note that by the definition of $u$, i.e. \eqref{e:def-u}, we have  $X=G>0$ at any point where $u>0$.

Now we prove the theorem in two cases:\\
\textbf{Case (1):} $n$, $p$ and $a(x,t)$  satisfy (T1).

In this case, $Z$ and $p$ satisfies the assumption (Z1) or (Z2) in Lemma~\ref{lem:H_estimate}. So we can find constants $\ep_1, \ep_2, \ep_3>0$ such that \eqref{e:lem1} holds.
Combining \eqref{ineq-G} and \eqref{e:lem1}, we get
\begin{equation}\label{ineq-G2}
\begin{aligned}
(\partial_t-\Delta )G \le  & -\ep_1X^2-\ep_2XY-\ep_3Z^2 + 2\lambda(n - 1)\kappa Y +2(\lambda -\delta p)e^{(p-1)g}\langle\nabla g, \nabla a\rangle \\
& + 2\langle\nabla g, \nabla X\rangle-\de e^{(p-1)g}\De a+(\de-1)e^{(p-1)g}\partial_t a.
\end{aligned}
\end{equation}
Since the function $a$ satisfies (A1), we can estimate the terms involving $a$ and its derivatives by
\begin{equation*}
|\delta| e^{(p-1)g} |\Delta a|\leq \frac{\ep_3}{4} Z^2+\frac{\de^2}{\ep_3}\frac{|\Delta a|^2}{a^2}\leq \frac{\ep_3}{4} Z^2+\frac{\de^2}{\ep_3}K^2,
\end{equation*}
\begin{equation*}
|\delta-1| e^{(p-1)g} |\partial_t a|\leq \frac{\ep_3}{4} Z^2+\frac{(1-\de)^2}{\ep_3}\frac{|\partial_t  a|^2}{a^2}\leq \frac{\ep_3}{4} Z^2+\frac{(1-\de)^2}{\ep_3}K^2,
\end{equation*}
and
\begin{equation*}
2|\lambda - \delta p|e^{(p-1)g} |\langle \nabla g, \nabla a \rangle|\leq \frac{\ep_3}{2} Z^2+\frac{2(\lambda-\de p)^2}{\ep_3}\frac{|\n a|^2}{a^2}|\n g|^2\leq \frac{\ep_3}{2} Z^2+\frac{2(\lambda-\de p)^2}{\ep_3}K|\n g|^2 .
\end{equation*}
Inserting the above inequalities into $\eqref{ineq-G2}$, we obtain
\begin{align*}
(\partial_t-\De)G \le &-\ep_1 G^2 - \ep_2|\nabla g|^2 G + \frac{\de^2+(1-\de)^2}{\ep_3}K^2+2\langle\n g,\n G\rangle\\
&+\left(2\lambda(n - 1)\kappa+\frac{2(\lambda-\de p)^2}{\ep_3}K\right) |\nabla g|^2.
\end{align*}
Therefore, by substituting $G$ by
\[
G = u+\frac{2\lambda(n-1)}{\ep_2}\kappa+\left(\frac{2(\lambda-\delta p)^2}{\ep_2\ep_3}+\sqrt{\frac{\de^2+(1-\de)^2}{\ep_1\ep_3}}\right)K,
\]
we obtain 
\begin{align*}
(\partial_t-\De)u &\le -\ep_1u^2 - \ep_2|\nabla g|^2 u - \sqrt{\frac{\de^2+(1-\de)^2}{\ep_1\ep_3}}\ep_2 K |\nabla g|^2+2\langle\n g,\n u\rangle\\
&\leq -\ep_1u^2 - \ep_2|\nabla g|^2 u+2\langle\n g,\n u\rangle.
\end{align*}
This proves \eqref{ineq-pt} under assumption (T1).

\

\noindent\textbf{Case (2):}  $n$, $p$ and $a(x,t)$ satisfy $(T2)$ or $(T3)$.

In this case, $n$ and $p$ satisfies the assumption (Z3) in Lemma~\ref{lem:H_estimate2}. So there exist constants $\lambda$, $\de=\frac{\ld}{p}$ and $\ep_3>0$ such that \eqref{e:lem2} holds for every $|\theta|<\ep_3$. Combining \eqref{ineq-G} and \eqref{e:lem2}, we get
\begin{equation}\label{ineq-G3}
(\partial_t-\Delta )G \le -\ep_1(\theta)X^2-\ep_2(\theta)XY-\theta XZ + 2\lambda(n - 1)\kappa Y 
+ 2\langle\nabla g, \nabla X\rangle -\frac{\ld}{p} e^{(p-1)g}L_\lambda a.
\end{equation}
where $L_\ld:=\De +(\frac{p}{\ld}-1)\partial_t$, $\ld=\frac{p}{2}$ if $n\ge 4$ and $\ld$ is sufficiently close to $1$ if $n\le 3$

Given a point $(x,t)\in P_{R,T}$ such that $u(x,t)>0$, we consider the following two cases depending on the sign of $L_\ld a(x,t)$.

\noindent\textbf{Case (2.1):} $L_\lambda a \geq 0.$ 

In this case, we simply set $\theta=0$. Then it follows from \eqref{ineq-G3} that
\begin{equation}\label{ineq-G4}
    \begin{aligned}
        (\partial_t-\Delta )G &\le   -\ep_1(0)X^2-\ep_2(0)XY + 2\lambda(n - 1)\kappa Y 
+ 2\langle\nabla g, \nabla X\rangle-\frac{\ld}{p} e^{(p-1)g}L_\lambda a\\
&\leq -\ep_1(0)X^2-\ep_2(0)XY + 2\lambda(n - 1)\kappa Y 
+ 2\langle\nabla g, \nabla X\rangle.
    \end{aligned}
\end{equation}

\noindent\textbf{Case (2.2):} $L_\lambda  a< 0$.

In this case, $a(x,t)$ must satisfy the assumption (A3) or (A5), hence $a(x,t)\neq 0$. If $a(x,t)< 0$, then $Z< 0$. According to (A3) or (A5), we have $L_\lambda a\geq Ka$. By setting $\theta=-\frac{\ep_3}{2}$, we deduce from \eqref{ineq-G3} that
\begin{equation}\label{ineq-G5}
    \begin{aligned}
        (\partial_t-\Delta )G &\le   -\ep_1(-\frac{\ep_3}{2})X^2-\ep_2(-\frac{\ep_3}{2})XY+\frac{\ep_3}{2}XZ-\frac{\ld}{p} ae^{(p-1)g}K+ 2\lambda(n - 1)\kappa Y 
+ 2\langle\nabla g, \nabla X\rangle\\
&\leq   -\ep_1(-\frac{\ep_3}{2})X^2-\ep_2(-\frac{\ep_3}{2})XY+Z(\frac{\ep_3}{2}X-\frac{\lambda}{p} K)+ 2\lambda(n - 1)\kappa Y 
+ 2\langle\nabla g, \nabla X\rangle
    \end{aligned}
\end{equation}

Otherwise, we have $a(x,t)>0$, hence $Z>0$. Again according to (A3) or (A5), we have  $L_\lambda a\geq -Ka$. By setting $\theta=\frac{\ep_3}{2}$, we deduce from \eqref{ineq-G3} that
\begin{equation}\label{ineq-G6}
    \begin{aligned}
        (\partial_t-\Delta )G &\le   -\ep_1(\frac{\ep_3}{2})X^2-\ep_2(\frac{\ep_3}{2})XY-\frac{\ep_3}{2}XZ+\frac{\ld}{p} ae^{(p-1)g}K+ 2\lambda(n - 1)\kappa Y 
+ 2\langle\nabla g, \nabla X\rangle\\
&\leq   -\ep_1(\frac{\ep_3}{2})X^2-\ep_2(\frac{\ep_3}{2})XY-Z(\frac{\ep_3}{2}X-\frac{\lambda}{p} K)+ 2\lambda(n - 1)\kappa Y 
+ 2\langle\nabla g, \nabla X\rangle
    \end{aligned}
\end{equation}

In conclusion, in either case, by setting 
\[\ep_1:=\min\left\{\ep_1(0),\ep_1(-\frac{\ep_3}{2}),\ep_1(\frac{\ep_3}{2})\right\}, \quad
\ep_2:=\min\left\{\ep_2(0),\ep_2(-\frac{\ep_3}{2}),\ep_2(\frac{\ep_3}{2})\right\}
\]
and substituting $G$ by
\[
X=G=u+\frac{2\lambda(n-1)}{\ep_2}\kappa+\frac{2\lambda}{p\ep_3}K,
\]
we obtain from \eqref{ineq-G4}, \eqref{ineq-G5} or \eqref{ineq-G6} that
\begin{align*}
(\partial_t-\De)u\leq -\ep_1u^2 - \ep_2|\nabla g|^2 u+2\langle\n g,\n u\rangle.
\end{align*}
This proves \eqref{ineq-pt} under assumption (T2) or (T3).
\end{proof}

\section{Integral inequality}\label{s:integral}

In this section, we derive an integral inequality from the key differential inequality \eqref{ineq-pt} in Theorem~\ref{l:pointwise-ineq1}, which is crucial to carry out the Nash-Moser iteration. We first recall the following Sobolev inequality due to Saloff-Coste~\cite{SC92}.

\begin{thm}\label{thm:sobolev}
Let $\left(M^n,g\right)$ be a complete Riemannian manifold with $\mathrm{Ric}\geq -\left(n-1\right)\kappa g$. For $n\ge 3$, there exists $C_n$ depending only on $n$, such that for every geodesic ball $B(o,R)\subset M$ and $f\in C_0^\infty(B\left(o,R\right))$, we have
\begin{equation*}
\left(\int_{B(o,R)}f^{2\chi} dV\right)^\frac{1}{\chi}\leq e^{C_n(1+\sqrt{\kappa}R)}V^{-\frac{2}{n}}R^{2}\int_{B(o,R)}\left(|\nabla f|^2+R^{-2}f^2\right)dV,
\end{equation*}
where $\chi=\frac{n}{n-2}$ and $V=vol(B(o,R))$. For $n\leq2$, the above inequality holds with $n$ replaced by any fixed $n'>2$.
\end{thm} 

To address the low-dimensional cases, we introduce an ``effective Sobolev dimension''
\[
N :=
\begin{cases}
n,&n\ge 3,\\
3,&n\le2,
\end{cases}
\]
and set
\[
\mu=\frac{N }{N -2},\qquad
q=\frac{2\mu-1}{\mu}=\frac{N +2}{N }.
\]
Then by Theorem~\ref{thm:sobolev}, for any dimension $n\ge 1$, we have a unified Sobolev inequality
\begin{equation}\label{e:Sobolev}
\left(\int_{B(o,R)}f^{2\mu} dV\right)^\frac{1}{\mu}\leq e^{C_n(1+\sqrt{\kappa}R)}V^{-\frac{2}{N}}R^{2}\int_{B(o,R)}\left(|\nabla f|^2+R^{-2}f^2\right)dV,
\end{equation}
which will be employed throughout this section.

For convenience, we denote $V := vol(B(o,R))$. By a translation in time, it suffices to work on $P_{R,T}:=B(o,R)\times[-T,0]$. The general cylinder
$B(o,R)\times[t_0-T,t_0]$ is recovered by replacing $t$ with
$t-t_0$. 

Suppose that $\phi(x)$ is a spatial cut-off function supported in $B(o,R)$ and $\Phi(t)$ is an increasing cut-off function on $[-T,0]$ such that $\Phi(-T)=0$ and $\Phi(t)=1$ for all $t\in [-T', 0]$ where $0<T'<T$. 

Setting $A:=\ep_1$ and $C:=\frac{1}{\ep_2}$, inequality \eqref{ineq-pt} can be rewritten as
\begin{equation}\label{ineq:AC}
(\partial_t-\De)u\leq -Au^2 - \frac{1}{C}|\nabla g|^2 u+2\langle\n g,\n u\rangle.
\end{equation}

\begin{lem}\label{lem:integral-ineq}
There exists $\beta_*\ge 1$, such that for all $\beta\ge \beta_*$, there holds
\begin{equation}\label{eq:l1}
\begin{aligned}
&18\int_{-T}^{0}\int_{B(o,R)} \Phi^2(t)|\nabla \phi|^2u^{2\beta}\,dV\,dt+6\int_{-T}^{0}\int_{B(o,R)}R^{-2}\Phi^2(t)\phi^2u^{2\beta}\,dV\,dt\\
&\quad +6\int_{-T}^{0}\int_{B(o,R)}\Phi(t)\Phi'(t)\phi^2u^{2\beta}\,dV\,dt\\
&\geq 3\beta A\int_{-T}^{0}\int_{B(o,R)}\Phi^2(t)\phi^2u^{2\beta+1}\,dV\,dt\\
&\quad +\bigg(e^{-C_n(1+\sqrt{\kappa}R)}R^{-2}V^\frac{2}{N }
\int_{-T'}^{0}\int_{B(o,R)}(\phi u^\beta)^{2\frac{2\mu-1}{\mu}}\,dV\,dt\bigg)^\frac{\mu}{2\mu-1}.
\end{aligned}
\end{equation}
\end{lem}

\begin{proof}
      For simplicity, we may assume $u\ge 0$ in $P_{R,T}$, otherwise we work with its positive part $u_+=\max\{u,0\}$.  

For any $\beta>1$, we multiply (\ref{ineq:AC}) by a test function $\phi^2u^{2\beta-1}$ and integrate over $B(o,R)$ to get
    \begin{align*}
    &\int_{B(o,R)}\phi^2u^{2\beta-1}\Delta u\,dx-\int_{B(o,R)}\phi^2u^{2\beta-1}u_t\,dx \geq \\&A\int_{B(o,R)}\phi^2u^{2\beta+1}\,dx+\frac{1}{C}\int_{B(o,R)}\phi^2 u^{2\beta}|\nabla g|^2-2\int_{B(o,R)}\phi^2u^{2\beta-1}\langle\n g,\n u\rangle\,dx.
    \end{align*}
    Applying integration by parts to the Laplacian term yields
    \[
    \begin{aligned}
    \int_{B(o,R)}\phi^2u^{2\beta-1}\Delta u\,dx &=-\int_{B(o,R)}\left\langle\nabla(\phi^2u^{2\beta-1}), \nabla u \right\rangle\,dx \\
    &=-2\int_{B(o,R)}\phi u^{2\beta-1}\langle\nabla\phi,\nabla u\rangle\,dx-(2\beta-1)\int_{B(o,R)}\phi^2u^{2\beta-2}|\nabla u|^2\,dx,
    \end{aligned}
    \]
    which implies that
    \begin{align*}
    &-2\int_{B(o,R)}\phi u^{2\beta-1}\langle \nabla \phi,\nabla u\rangle\,dx \geq \, A\int_{B(o,R)}\phi^2u^{2\beta+1}\,dx+(2\beta-1)\int_{B(o,R)}\phi^2u^{2\beta-2}|\nabla u|^2\,dx\\
    &+\frac{1}{2\beta}\int_{B(o,R)}\phi^2(u^{2\beta})_t\,dx+\frac{1}{C}\int_{B(o,R)}\phi^2 u^{2\beta}|\nabla g|^2-2\int_{B(o,R)}\phi^2u^{2\beta-1}\langle\n g,\n u\rangle\,dx.
    \end{align*}
    By Young's inequality, we can bound the last term by
    \[
    2\phi^2u^{2\beta-1}\langle\n g,\n u\rangle\leq \frac{1}{C}\phi^2u^{2\beta}|\n g|^2+C\phi^2u^{2\beta-2}|\n u|^2.
    \]
    It follows that
    \begin{align*}
-2\int_{B(o,R)}\phi u^{2\beta-1}\langle \nabla \phi,\nabla u\rangle\,dV \geq &\, A\int_{B(o,R)}\phi^2u^{2\beta+1}\,dV+(2\beta-1-C)\int_{B(o,R)}\phi^2u^{2\beta-2}|\nabla u|^2\,dV\\
&+\frac{1}{2\beta}\int_{B(o,R)}\phi^2(u^{2\beta})_t\,dV.
\end{align*}
    By Cauchy-Schwarz and Young's inequalities, we have
    \begin{align*}
    -2\int_{B(o,R)}\phi u^{2\beta-1}\langle \nabla \phi,\nabla u\rangle\,dV \leq &\, \frac{(2\beta-1-C)}{2}\int_{B(o,R)}\phi^2 u^{2\beta-2}|\nabla u|^2\,dV\\
    &+\frac{2}{2\beta-1-C}\int_{B(o,R)}|\nabla \phi|^2u^{2\beta}\,dV.
    \end{align*}
    Combining this with
    \begin{equation*}
    \int_{B(o,R)}|\nabla(\phi u^\beta)|^2\,dV \leq 2\int_{B(o,R)}|\nabla \phi|^2u^{2\beta}\,dV+2\beta^2\int_{B(o,R)}\phi^2 u^{2\beta-2}|\nabla u|^2\,dV,
    \end{equation*}
    we deduce that
    \begin{align*}
    &\bigg(\frac{2}{2\beta-1-C}+\frac{2\beta-1-C}{2\beta^2}\bigg)\int_{B(o,R)}|\nabla \phi|^2u^{2\beta}\,dV\\
    &\geq A\int_{B(o,R)}\phi^2u^{2\beta+1}\,dV+\frac{2\beta-1-C}{4\beta^2}\int_{B(o,R)}|\nabla(\phi u^\beta)|^2\,dV+\frac{1}{2\beta}\int_{B(o,R)}\phi^2(u^{2\beta})_t\,dV.
\end{align*}

Now applying the Saloff-Coste's Sobolev inequality~\eqref{e:Sobolev} to the gradient term gives
\begin{align*}
&\bigg( \frac{2}{2\beta-1-C}+\frac{2\beta-1-C}{2\beta^2}\bigg)\int_{B(o,R)}|\nabla \phi|^2u^{2\beta}\,dV+\frac{2\beta-1-C}{4\beta^2}R^{-2}\int_{B(o,R)}\phi^2u^{2\beta}\,dV\\
&\geq  A\int_{B(o,R)}\phi^2u^{2\beta+1}\,dV+\frac{2\beta-1-C}{4\beta^2}e^{-C_n(1+\sqrt{\kappa}R)}R^{-2}V^\frac{2}{N }\bigg(\int_{B(o,R)}(\phi u^\beta)^{2\mu}\,dV\bigg)^\frac{1}{\mu}\\
&\quad +\frac{1}{2\beta}\int_{B(o,R)}\phi^2(u^{2\beta})_t\,dV.
\end{align*}
By choosing $\beta_*\geq 2(1+C)$ sufficiently large, we get that for every $\beta\geq \beta_*$ the following holds
\begin{equation}\label{int-eq1}
\begin{aligned}
&\frac{3}{\beta}\int_{B(o,R)} |\nabla \phi|^2u^{2\beta}\,dV+\frac{1}{2\beta}R^{-2}\int_{B(o,R)}\phi^2u^{2\beta}\,dV\\
&\geq A\int_{B(o,R)}\phi^2u^{2\beta+1}\,dV+\frac{1}{3\beta}e^{-C_n(1+\sqrt{\kappa}R)}R^{-2}V^\frac{2}{N }\bigg(\int_{B(o,R)}(\phi u^\beta)^{2\mu}\,dV\bigg)^\frac{1}{\mu}\\
&\quad +\frac{1}{2\beta}\int_{B(o,R)}\phi^2(u^{2\beta})_t\,dV.
\end{aligned}
\end{equation}

Next, for any $t'\in (-T',\, 0)$, multiplying \eqref{int-eq1} by $\Phi^2(t)$ and integrating over $[-T,t']$ gives
\begin{align*}
&\frac{3}{\beta}\int_{-T}^{t'}\int_{B(o,R)} \Phi^2(t)|\nabla \phi|^2u^{2\beta}\,dV\,dt+\frac{1}{2\beta}\int_{-T}^{t'}\int_{B(o,R)}\big(R^{-2}\Phi^2(t)\phi^2u^{2\beta}+ 2\Phi(t)\Phi'(t)\phi^2 u^{2\beta}\big)\,dV\,dt\\
&\geq A\int_{-T}^{t'}\int_{B(o,R)}\Phi^2(t)\phi^2u^{2\beta+1}\,dV\,dt+\frac{1}{3\beta}e^{-C_n(1+\sqrt{\kappa}R)}R^{-2}V^\frac{2}{N }\int_{-T}^{t'}\Phi^2(t)\bigg(\int_{B(o,R)}(\phi u^\beta)^{2\mu}\,dV\bigg)^\frac{1}{\mu}dt\\
&+\frac{1}{2\beta}\int_{B(o,R)}\Phi^2(t')\phi^2 u^{2\beta}(x, t')\,dV.
\end{align*}
Since $\Phi(t)$ is an increasing function and $\Phi(t) = 1$ for all $t\in [-T',\, 0]$, the preceding inequality implies that for every $t'\in [-T',\, 0]$,
\begin{equation}\label{eq-4}
\begin{aligned}
&\frac{3}{\beta}\int_{-T}^{0}\int_{B(o,R)} \Phi^2(t)|\nabla \phi|^2u^{2\beta}\,dV\,dt +\frac{1}{2\beta}\int_{-T}^{0}\int_{B(o,R)}\big(R^{-2}\Phi^2(t)\phi^2u^{2\beta}+ 2\Phi(t)\Phi'(t)\phi^2 u^{2\beta}\big)\,dV\,dt\\
&\geq A\int_{-T}^{t'}\int_{B(o,R)}\Phi^2(t)\phi^2u^{2\beta+1}\,dV\,dt+\frac{1}{3\beta}e^{-C_n(1+\sqrt{\kappa}R)}R^{-2}V^\frac{2}{N }\int_{-T}^{t'}\Phi^2(t)\bigg(\int_{B(o,R)}(\phi u^\beta)^{2\mu}\,dV\bigg)^\frac{1}{\mu}dt\\
&\quad +\frac{1}{2\beta}\int_{B(o,R)}\Phi^2(t')\phi^2 u^{2\beta}(x, t')\,dV.
\end{aligned}
\end{equation}

To simplify the notation, we denote the left-hand side of the above inequality by
\[
I:=\frac{3}{\beta}\int_{-T}^{0}\int_{B(o,R)} \Phi^2(t)|\nabla \phi|^2u^{2\beta}\,dV\,dt +\frac{1}{2\beta}\int_{-T}^{0}\int_{B(o,R)}\big(R^{-2}\Phi^2(t)\phi^2u^{2\beta}+ 2\Phi(t)\Phi'(t)\phi^2 u^{2\beta}\big)dVdt.
\]
By dropping appropriate nonnegative terms on the right-hand side and letting $t'=0$, we obtain the following two inequalities:
\begin{equation}\label{eq-1}
\begin{aligned}
&I\geq A\int_{-T}^{0}\int_{B(o,R)}\Phi^2(t)\phi^2u^{2\beta+1}\,dV\,dt
\end{aligned}
\end{equation}
and
\begin{equation}\label{eq-2}
\begin{aligned}
&I\geq \frac{1}{3\beta}e^{-C_n(1+\sqrt{\kappa}R)}R^{-2}V^\frac{2}{N }\int_{-T'}^{0}\bigg(\int_{B(o,R)}(\phi u^\beta)^{2\mu}\,dV\bigg)^\frac{1}{\mu}\,dt.
\end{aligned}
\end{equation}
Alternatively, by taking the supremum over $t'\in [-T',\, 0]$ in \eqref{eq-4} and dropping the integral terms on the right-hand side, we obtain
\begin{equation}\label{eq-22}
    I\geq \sup_{-T'\leq t'\leq 0}\frac{1}{2\beta}\int_{B(o,R)}\phi^2 u^{2\beta}(x, t')\,dV \geq \frac{1}{3\beta}\sup_{-T'\leq t'\leq 0}\int_{B(o,R)}\phi^2 u^{2\beta}(x, t')\,dV.
\end{equation}

Using \eqref{eq-2} and \eqref{eq-22}, together with the H\"older inequality, we get
\begin{equation}\label{eq-3}
\begin{aligned}
I^{\frac{2\mu-1}{\mu}}&\geq \bigg(\frac{1}{3\beta}\bigg)^{\frac{2\mu-1}{\mu}}e^{-C_n(1+\sqrt{\kappa}R)}R^{-2}V^\frac{2}{N }\bigg(\sup_{-T'\leq t'\leq 0}\int_{B(o,R)}\phi^2 u^{2\beta}(x, t')\,dV\bigg)^\frac{\mu-1}{\mu}\\
&\cdot\int_{-T'}^{0}\bigg(\int_{B(o,R)}(\phi u^\beta)^{2\mu}\,dV\bigg)^\frac{1}{\mu}\,dt\\
&\geq \bigg(\frac{1}{3\beta}\bigg)^{\frac{2\mu-1}{\mu}}e^{-C_n(1+\sqrt{\kappa}R)}R^{-2}V^\frac{2}{N }\int_{-T'}^{0}\bigg(\int_{B(o,R)}\phi^2 u^{2\beta}(x, t)\,dV\bigg)^\frac{\mu-1}{\mu}\bigg(\int_{B(o,R)}(\phi u^\beta)^{2\mu}\,dV\bigg)^\frac{1}{\mu}\,dt\\
&\geq \bigg(\frac{1}{3\beta}\bigg)^{\frac{2\mu-1}{\mu}}e^{-C_n(1+\sqrt{\kappa}R)}R^{-2}V^\frac{2}{N }\int_{-T'}^{0}\int_{B(o,R)}(\phi u^\beta)^{2\frac{2\mu-1}{\mu}}\,dV\,dt.
\end{aligned}
\end{equation}

Combining inequalities \eqref{eq-1} and \eqref{eq-3}, we conclude that
\begin{align*}
&\frac{6}{\beta}\int_{-T}^{0}\int_{B(o,R)} \Phi^2(t)|\nabla \phi|^2u^{2\beta}\,dV\,dt +\frac{2}{\beta}\int_{-T}^{0}\int_{B(o,R)}\big(R^{-2}\Phi^2(t)\phi^2u^{2\beta} + \Phi(t)\Phi'(t)\phi^2u^{2\beta}\big)\,dV\,dt\geq 2I\\
&\geq A\int_{-T}^{0}\int_{B(o,R)}\Phi^2(t)\phi^2u^{2\beta+1}\,dV\,dt+\frac{1}{3\beta}\bigg(e^{-C_n(1+\sqrt{\kappa}R)}R^{-2}V^\frac{2}{N }\int_{-T'}^{0}\int_{B(o,R)}(\phi u^\beta)^{2\frac{2\mu-1}{\mu}}\,dV\,dt\bigg)^\frac{\mu}{2\mu-1}.
\end{align*}
Multiplying the inequality by $3\beta$, we obtain $\eqref{eq:l1}$. This completes the proof.
\end{proof}

Next, we derive an initial $L^p$-bound, which will serve as the starting point for the Nash--Moser iteration. Set 
\begin{equation}\label{e:def-beta0}
   \beta_0 :=
\max\left\{
\beta_*,
C_0\left(
1+\sqrt\kappa R+\frac{T}{R^2}+\frac{R^2}{T}
\right)
\right\}, 
\end{equation}
where $\beta_*$ is given by Lemma~\ref{lem:integral-ineq} and $C_0>0$ is a constant to be determined later. Furthermore, we define
$$\beta_1 := \frac{2\mu-1}{\mu}\beta_0=\frac{N +2}{N }\beta_0.$$

\begin{lem}\label{lem:initial-bound}
There exists $C_0>0$ such that
\begin{equation}\label{lem-4.1}
\begin{aligned}
    \bigg(e^{-C_n(1+\sqrt{\kappa}R)}R^{-2}V^{\frac{2}{N }}\int_{-\frac{3T}{4}}^{0}\int_{B(o,\frac{3}{4}R)}&u^{2\beta_1}\,dV\,dt\bigg)^\frac{1}{2\beta_1}\\
    &\leq C(n,p,\de,\lambda)\bigg(\Big(\frac{(2\beta_0+1)^{2\beta_0+1}}{R^{4\beta_0+2}}+\frac{1}{T^{2\beta_0+1}}\Big)TV\bigg)^{\frac{1}{2\beta_0}}.
\end{aligned}
\end{equation}
\end{lem}

\begin{proof}
We split the domain of integration by
\[
\Omega_1=\left\{ (x,t)\in P_{R,T}\,\bigg|\, u\geq \frac{4R^{-2}}{A\beta_0}\right\}, \quad \Omega_2=P_{R,T}\setminus\Omega_1.
\]
Then, estimating the integrals over $\Omega_1$ and $\Omega_2$ separately, we obtain
\begin{equation}\label{eq:l2-0}
\begin{aligned}
&\int_{-T}^{0}\int_{B(o,R)}R^{-2}\Phi^2(t)\phi^2u^{2\beta_0}\,dV\,dt\\
&\leq \int_{\Omega_1}R^{-2}\Phi^2(t)\phi^2u^{2\beta_0}\,dV\,dt+\int_{\Omega_2}R^{-2}\Phi^2(t)\phi^2u^{2\beta_0}\,dV\,dt\\
&\leq \frac{1}{4}\beta_0 A \int_{\Omega_1}\Phi^2(t)\phi^2u^{2\beta_0+1}\,dV\,dt
+\int_{-T}^{0}\int_{B(o,R)}\frac{4^{2\beta_0}(R^{-2})^{2\beta_0+1}}{A^{2\beta_0}\beta_0^{2\beta_0}}\Phi^2(t)\phi^2\,dV\,dt.
\end{aligned}
\end{equation}

Now, choose $\phi_1$ be a standard flat cut-off function on $C_0^{\infty}(B(o, R))$ such that $0 \leq \phi_1 \leq 1$, $\phi_1=1$ on $B\left(o, \frac{3}{4} R\right)$, and $\left|\nabla \phi_1\right| \leq \frac{8}{R}$. Since $\phi_1$ vanishes to infinite order at the boundary of its zero set, the function $\phi_1^{2\beta_0+1}$
 is smooth. 
 
 Let $\phi=\phi_1^{2\beta_0+1}$. A direct computation shows
$$
R^2|\nabla \phi|^2 \leq 64(2\beta_0+1)^2 \phi^{\frac{4 \beta_0}{2\beta_0+1}}.
$$

 We choose $\eta$ to be a standard flat cut-off function on $C^{\infty}([-T,\,0 ])$ such that $0 \leq \eta \leq 1$, $\eta=0$ on $[-T,\, -\frac{15}{16}T]$, $\eta=1$ on $[-\frac{3T}{4},\, 0]$, and $0\leq \eta' \leq \frac{8}{T}$. Since $\eta$ vanishes to infinite order at the boundary of its zero set, the function $\eta^{\frac{2\beta_0+1}{2}}$
 is smooth. 
 
 Let $\Phi(t)=\eta^{\frac{2\beta_0+1}{2}}$. Another direct computation gives
$$
T \Phi'(t) \leq 16(2\beta_0+1)  \Phi^{\frac{2\beta_0-1}{2\beta_0+1}}.
$$
By Young's inequality, there exists a universal constant $C>0$ (which may change from line to line) such that
$$
\begin{aligned}
R^2 \int_{B(o,R)} |\nabla \phi|^2 u^{2\beta_0}\,dV
& \leq C (2\beta_0+1)^2 \int_{B(o,R)} \phi^{\frac{4\beta_0}{2\beta_0+1}}u^{2\beta_0}\,dV\\
& \leq C (2\beta_0+1)^2\left(\int_{B(o,R)} \phi^2 u^{2\beta_0+1}\,dV\right)^{\frac{2\beta_0}{2\beta_0+1}} V^{\frac{1}{2\beta_0+1}}\\
& \leq \frac{1}{16} \beta_0 R^2 A \int_{B(o,R)}\phi^2 u^{2\beta_0+1}\,dV + \frac{16^{2\beta_0}C^{2\beta_0+1}(2\beta_0+1)^{2\beta_0+1}}{R^{4\beta_0}A^{2\beta_0}} V.
\end{aligned}
$$
It follows that
\begin{equation}\label{eq:l2-1}
\begin{aligned}
18\int_{-T}^{0}\int_{B(o,R)} \Phi^2|\nabla \phi|^2 u^{2\beta_0}\,dV\,dt
&\leq \frac{9}{8} \beta_0 A \int_{-T}^{0}\int_{B(o,R)}\Phi^2\phi^2 u^{2\beta_0+1}\,dV\,dt\\
&\quad +18\int_{-T}^{0}\Phi^2\frac{16^{2\beta_0}C^{2\beta_0+1}(2\beta_0+1)^{2\beta_0+1}}{A^{2\beta_0}R^{2(2\beta_0+1)}} V\,dt\\
&\leq \frac{9}{8}\beta_0 A\int_{-T}^{0}\int_{B(o,R)}\Phi^2\phi^2 u^{2\beta_0+1}\,dV\,dt\\
&+ \frac{C^{2\beta_0+1}(2\beta_0+1)^{2\beta_0+1}T}{A^{2\beta_0}R^{4\beta_0+2}} V.
\end{aligned}
\end{equation}

Meanwhile, similar estimates for the time derivative yield
$$
\begin{aligned}
T\int_{-T}^{0} \Phi\Phi' u^{2\beta_0}\,dt
& \leq C (2\beta_0+1) \int_{-T}^{0} \Phi^{\frac{4\beta_0}{2\beta_0+1}}u^{2\beta_0}\,dt\\
& \leq C (2\beta_0+1)\left(\int_{-T}^{0} \Phi^2 u^{2\beta_0+1}\,dt\right)^{\frac{2\beta_0}{2\beta_0+1}} T^{\frac{1}{2\beta_0+1}}      \\
& \leq \frac{1}{16} \beta_0 T A \int_{-T}^{0}\Phi^2 u^{2\beta_0+1}\,dt + \frac{16^{2\beta_0}C^{2\beta_0+1}}{A^{2\beta_0}T^{2\beta_0}} T.
\end{aligned}
$$
Hence
\begin{equation}\label{eq:l2-2}
\begin{aligned}
6\int_{-T}^{0}\int_{B(o,R)} \phi^2 \Phi\Phi'u^{2\beta_0}\,dV\,dt
&\leq \frac{3}{8}\beta_0 A\int_{-T}^{0}\int_{B(o,R)}\Phi^2\phi^2 u^{2\beta_0+1}\,dV\,dt\\
&+6\int_{B(o,R)}\phi^2\frac{16^{2\beta_0}C^{2\beta_0+1}}{A^{2\beta_0}T^{2\beta_0}}\,dV\\
&\leq \frac{3}{8}\beta_0 A\int_{-T}^{0}\int_{B(o,R)}\Phi^2\phi^2 u^{2\beta_0+1}\,dV\,dt +\frac{C^{2\beta_0+1}}{A^{2\beta_0}T^{2\beta_0}}V.
\end{aligned}
\end{equation}

Inserting \eqref{eq:l2-0}, \eqref{eq:l2-1}, and \eqref{eq:l2-2} into \eqref{eq:l1}, we get
\begin{align*}
&\bigg(e^{-C_n(1+\sqrt{\kappa}R)}R^{-2}V^\frac{2}{N }\int_{-\frac{3T}{4}}^{0}\int_{B(o,\frac{3}{4}R)}(u^{2\beta_0})^{\frac{2\mu-1}{\mu}}\,dV\,dt\bigg)^\frac{\mu}{2\mu-1}\\
&\leq \frac{4^{2\beta_0}(R^{-2})^{2\beta_0+1}T}{A^{2\beta_0}\beta_0^{2\beta_0}}V+\frac{C^{2\beta_0+1}(2\beta_0+1)^{2\beta_0+1}T}{A^{2\beta_0}R^{4\beta_0+2}} V+\frac{4^{2\beta_0}C^{2\beta_0+1}}{A^{2\beta_0}T^{2\beta_0}} V\\
&\le \frac{C^{2\beta_0+1}(2\beta_0+1)^{2\beta_0+1}}{A^{2\beta_0}}TV\frac{1}{R^{4\beta_0+2}}
+\frac{C^{2\beta_0+1}}{A^{2\beta_0}}TV\frac{1}{T^{2\beta_0+1}}.
\end{align*}
Therefore, recalling that $\beta_1 = \frac{2\mu-1}{\mu}\beta_0$, we conclude the desired inequality \eqref{lem-4.1}.
\end{proof}

\section{Proof of Main Theorems}\label{s:proof}

We are now in a position to establish the Li-Yau gradient estimates in our main theorems via the Nash--Moser iteration. 

\begin{proof}[Proof of Theorem~\ref{thm:main1}]
By assumption, $n\ge 3$, $p$ satisfies (P1) or (P2), and $a(x,t)$ satisfies (A1), which fulfills the assumption (T1) in Theorem~\ref{l:pointwise-ineq1}. So the function $u$ defined by \eqref{e:def-u} satisfies the pointwise differential inequality \eqref{ineq-pt}. Then by Lemma \ref{lem:integral-ineq}, there exists $\beta_*\ge 1$, such that for all $\beta\ge \beta_*$, $v>0$ and $-T<s+v<0$, the following inequality holds:
\begin{equation}\label{ineq-moser}
\begin{aligned}
&18\int_{s}^{0}\int_{B(o,R)}u^{2\beta}\big(\Phi^2(t)|\nabla \phi|^2+R^{-2}\Phi^2(t)\phi^2+\Phi(t)\Phi'(t)\phi^2\big)\,dV\,dt\\ 
&\geq \bigg(e^{-C_n(1+\sqrt{\kappa}R)}R^{-2}V^\frac{2}{N }\int_{s+v}^{0}\int_{B(o,R)}(\phi u^\beta)^{2\frac{2\mu-1}{\mu}}\,dV\,dt \bigg)^\frac{\mu}{2\mu-1}.
\end{aligned}
\end{equation}

Let $\beta_0$ be defined by \eqref{e:def-beta0}. We set up the iterative sequences as follows:
\begin{align*}
&\beta_{k}=\beta_{k-1}\frac{2\mu-1}{\mu}, \quad r_k=\frac{R}{2}+\frac{R}{4^k},\\
&s_k=-\frac{T}{2}-\frac{T}{2^{k+1}}, \quad \text{and} \quad v_k=\frac{T}{2^{k+2}}, \quad \text{for } k=1, 2, \cdots.
\end{align*}
Let $\phi_k\in C_0^\infty(B(o,r_k))$ and $\Phi_k\in C^\infty([s_k,0])$ be cut-off functions satisfying
\begin{equation*}
\phi_k\equiv 1 \quad \text{on } B(o,r_{k+1}), \quad \text{with} \quad 0\leq \phi_k\leq 1 \quad\text{and} \quad |\nabla \phi_k|\leq C\frac{4^{k+1}}{R},
\end{equation*}
and
\begin{equation*}
\Phi_k(s_k)=0\quad \text{and}\quad\Phi_k\equiv 1 \quad \text{on } [s_k+v_k,0], \quad \text{with} \quad 0\leq \Phi_k\leq 1 \quad \text{and}\quad 0\leq \Phi_k'\leq C\frac{8^{k+1}}{T}.
\end{equation*}
For each $k$, substituting $\beta=\beta_k$, $\phi=\phi_k$, and $\Phi=\Phi_k$ into \eqref{ineq-moser}, we arrive at
\begin{align*}
\|u\|_{2\beta_{k+1},r_{k+1},[s_{k+1},0]}
&\leq \big(Ce^{C_R}R^2V^{-\frac{2}{N }}\big)^\frac{1}{2\beta_{k+1}}(R^{-2}+T^{-1})^{\frac{1}{2\beta_k}}\left(16^{k+1}+8^{k+1}\right)^\frac{1}{2\beta_k}\|u\|_{2\beta_{k},r_{k},[s_{k},0]}\\
&\le  C^\frac{1}{2\beta_k}\big(e^{C_R}R^2V^{-\frac{2}{N }}\big)^\frac{1}{2\beta_{k+1}}(R^{-2}+T^{-1})^{\frac{1}{2\beta_k}}16^\frac{k}{2\beta_k}\|u\|_{2\beta_{k},r_{k},[s_{k},0]},
\end{align*}
where we have denoted
\begin{equation*}
\|u\|_{2\beta,r,[s,t]}:=\bigg(\int_s^t\int_{B(o,r)}u^{2\beta}\,dV\,dt\bigg)^\frac{1}{2\beta} \quad \text{and} \quad C_R=C_n(1+\sqrt{\kappa}R).
\end{equation*}
By iterating the above inequality, we obtain
\begin{equation}\label{eq:iteration}
\|u\|_{2\beta_{k+1},r_{k+1},[s_{k+1},0]}\leq C^{\sum_{i=1}^{k+1}\frac{1}{2\beta_i}}\big(e^{C_R}R^2V^{-\frac{2}{N }}\big)^{\sum_{i=2}^{k+1}\frac{1}{2\beta_i}}(R^{-2} +T^{-1})^{\sum_{i=1}^{k}\frac{1}{2\beta_i}}16^{\sum_{i=1}^{k}\frac{i}{2\beta_i}}\|u\|_{2\beta_{1},r_{1},[s_{1},0]}.
\end{equation}
Notice that the series converge as
\[
\sum_{i=1}^\infty\frac{1}{2\beta_i}=\frac{N +2}{4\beta_1}= \frac{N }{4\beta_0},\quad \sum_{i=2}^\infty\frac{1}{2\beta_i}=\frac{N +2}{4\beta_2}= \frac{N ^2}{4\beta_0(N +2)}, \quad \sum_{i=1}^{\infty}\frac{i}{2\beta_i}=\frac{N (N +2)}{8\beta_0}. 
\]
Thus, letting $k\rightarrow \infty$ in \eqref{eq:iteration} yields
\begin{equation}\label{eq-5}
\|u\|_{\infty,\frac{1}{2}R,[-\frac{T}{2},0]}\leq  C^{\frac{N}{4\beta_0}}\big(e^{C_R}R^2V^{-\frac{2}{N }}\big)^{\frac{N ^2}{4\beta_0(N +2)}}(R^{-2}+T^{-1})^\frac{N }{4\beta_0}
16^{\frac{N (N +2)}{8\beta_0}}\|u\|_{2\beta_{1},r_{1},[s_{1},0]}.
\end{equation}
Substituting the initial $L^p$-bound \eqref{lem-4.1} from Lemma~\ref{lem:initial-bound} into \eqref{eq-5}, we get
\begin{align*}
\|u\|_{\infty,\frac{1}{2}R,[-\frac{T}{2},0]} \leq& \, C(e^{C_R})^{\frac{N }{4\beta_0}}(R^2V^{-\frac{2}{N }})^{\frac{N ^2}{4\beta_0(N +2)}}(R^{-2}+T^{-1})^\frac{N }{4\beta_0}V^{\frac{1}{2\beta_0}-\frac{2}{N }\frac{\mu}{2\beta_0(2\mu-1)}}R^{\frac{2\mu}{2\beta_0(2\mu-1)}}\\
&\cdot\left(\left(\frac{(2\beta_0+1)^{2\beta_0+1}}{R^{4\beta_0+2}}+\frac{1}{T^{2\beta_0+1}}\right)T\right)^{\frac{1}{2\beta_0}}\\
\leq& \, CR^{\frac{N ^2}{2\beta_0(N +2)}+\frac{2N }{2\beta_0(N +2)}}(R^{-2}+T^{-1})^\frac{N }{4\beta_0}
V^{-\frac{N }{2\beta_0(N +2)}+\frac{1}{2\beta_0}-\frac{2}{2\beta_0(N +2)}}\\
&\cdot\left(\left(\frac{(2\beta_0+1)^{2\beta_0+1}}{R^{4\beta_0+2}}+\frac{1}{T^{2\beta_0+1}}\right)T\right)^{\frac{1}{2\beta_0}}.
\end{align*}
Through a straightforward computation and rearrangement, this simplifies to
\begin{align*}
\|u\|_{\infty,\frac{1}{2}R,[-\frac{T}{2},0]}\leq & \, CR^\frac{N}{2\beta_0}(R^{-2}+T^{-1})^\frac{N }{4\beta_0}T^{\frac{1}{2\beta_0}}\left(\frac{(2\beta_0+1)^{2\beta_0+1}}{R^{4\beta_0+2}}+\frac{1}{T^{2\beta_0+1}}\right)^\frac{1}{2\beta_0}\\
= & \, CR^{-2}\bigg[\left(1+\frac{R^2}{T}\right)^{\frac{N }{2}}\left(\frac{T(2\beta_0+1)^{2\beta_0+1}}{R^2}+\left(\frac{R^2}{T}\right)^{2\beta_0}\right)\bigg]^\frac{1}{2\beta_0}\\
\leq & \, CR^{-2}\bigg[\left(1+\frac{R^2}{T}\right)^{\frac{N }{2}}\left(\beta_0(2\beta_0+1)^{2\beta_0+1}+\left(\frac{R^2}{T}\right)^{2\beta_0}\right)\bigg]^\frac{1}{2\beta_0}\\
\leq & \, C\left(\frac{2\beta_0+1}{R^2}+\frac{1}{T}\right) \leq C\left(\frac{T}{R^4}+\frac{1}{T}+\frac{1+\sqrt{\kappa}R}{R^2}\right),
\end{align*}
where in the last step we have used \eqref{e:def-beta0}.

Recalling the definition \eqref{e:def-u} of $u$, we immediately deduce that on $P_{R/2,T/2}$, 
\begin{equation}\label{e:final}
\begin{aligned}
&\lambda\frac{|\nabla f|^2}{f^2} -  \frac{\partial_t f}{f}+\delta a(x,t)f^{p-1}\leq \bigg(\lambda\frac{|\nabla f|^2}{f^2} -  \frac{\partial_t f}{f}+\delta a(x,t)f^{p-1}\bigg)_+\\
\leq & \, C(n,p,\lambda,\delta)\left[\frac{T}{R^4}+\frac{1}{T}+\frac{1+\sqrt{\kappa}R}{R^2}+\frac{2\lambda(n-1)}{\ep_2}\kappa+\left(\frac{2(\lambda-\delta p)^2}{\ep_2\ep_3}+\sqrt{\frac{\delta^2+(1-\de)^2}{\ep_1\ep_3}}\right)K\right].
\end{aligned}
\end{equation}

Now if $T\leq 2R^2$, \eqref{e:final} directly gives
\begin{align*}
&\lambda\frac{|\nabla f|^2}{f^2} -  \frac{\partial_t f}{f}+\delta a(x,t)f^{p-1}\\
\leq & \, C(n,p,\lambda,\delta)\left[\frac{1}{T}+\frac{1+\sqrt{\kappa} R}{R^2}+\frac{2\lambda(n-1)}{\ep_2}\kappa
+\left(\frac{2(\lambda-\delta p)^2}{\ep_2\ep_3}+\sqrt{\frac{\delta^2+(1-\de)^2}{\ep_1\ep_3}}\right)K\right]\\
\leq & \, C(n,p,\lambda,\delta)\left(\frac{1}{T}+\frac{1}{R^2}+\kappa+K\right)
\end{align*}
on $P_{R/2,T/2}$.
Otherwise if $T\geq 2R^2$, then for an arbitrary point $(x_1, t_1)\in P_{R/2,T/2}=B(o,\frac{R}{2})\times[-\frac{T}{2},0] $, we can apply \eqref{e:final} on the cylinder $P_{R,R^2}(t_1)=B(o,R)\times [t_1-R^2,\,t_1]$ to obtain
\begin{align*}
&\lambda\frac{|\nabla f|^2}{f^2} - \frac{\partial_t f}{f}+\delta a(x,t)f^{p-1}\\
\leq & \, C(n,p,\lambda,\delta)\left[\frac{1+\sqrt{\kappa}R}{R^2} +\frac{2\lambda(n-1)}{\ep_2}\kappa+\left(\frac{2(\lambda-\delta p)^2}{\ep_2\ep_3}+\sqrt{\frac{\delta^2+(1-\de)^2}{\ep_1\ep_3}}\right)K\right]\\
\leq & \, C(n,p,\lambda,\delta)\left(\frac{1+\kappa R^2}{R^2}+K\right) \leq C(n,p,\lambda,\delta)\left(\frac{1 }{R^2}+\frac{1}{T}+\kappa+K\right)
\end{align*}
in $P_{\frac{R}{2},\frac{R^2}{2}}(t_1)=B(o,\frac{R}{2})\times[t_1-\frac{R^2}{2},\,t_1]$.

Finally, recall that by Lemma~\ref{lem:H_estimate}, we have $\de>0$ when $a(x,t)$ is positive and $\de<0$ when $a(x,t)$ is negative. In either case, we may set $\de_0 = |\de|>0$ such that 
\[
    \lambda\frac{|\nabla f|^2}{f^2} -  \frac{\partial_t f}{f}+\delta_0 |a(x,t)|f^{p-1}\leq C(n,p,\lambda,\delta_0)\left(\frac{1}{T}+\frac{1}{R^2}+\kappa+K\right).
\]
This completes the proof.
\end{proof}

\begin{proof}[Proof of Theorem~\ref{thm:main2} and Theorem~\ref{thm:main3}]
Under the assumption of Theorem~\ref{thm:main2} or Theorem~\ref{thm:main3}, $n$, $p$ and $a(x,t)$ fulfill the assumption (T2) or (T3) in Theorem~\ref{l:pointwise-ineq1}, respectively.
Therefore, by setting $\delta=\lambda/p$, the function $u$ given by \eqref{e:def-u} satisfies the key differential inequality \eqref{ineq-pt}, where $\ld=\frac{p}{2}$ in case (T2) and $\ld$ is sufficiently close to $1$ in case (T3).

Then the same argument as in the proof of Theorem~\ref{thm:main1} yields the desired inequality \eqref{ineq:zhuyao2} in Theorem~\ref{thm:main2} and \eqref{ineq:zhuyao3} in Theorem~\ref{thm:main3}.
\end{proof}

Finally, we use Theorem~\ref{thm:main1},  Theorem~\ref{thm:main2} and Theorem~\ref{thm:main3} to establish the Liouville theorems.

\begin{proof}[Proof of Corollary~\ref{thm:liouville1}]
In case (1), suppose for contradiction that there exists a positive eternal solution $f$ to \eqref{e:main} on $M\times(-\infty, +\infty)$. By Theorem~\ref{thm:main1}, there exist constants $0<\lambda<1$ and $\delta_0>0$ such that on $P_{R/2,T/2}$, 
\begin{equation*}
\lambda\frac{|\nabla f|^2}{f^2}-\frac{f_t}{f}+\delta_0 |c|f^{p-1}\leq C(n,p,\lambda,\delta_0)\left(\frac{1}{T}+\frac{1}{R^2}\right).
\end{equation*}
Letting $T\rightarrow+\infty$ and $R\rightarrow +\infty$, we find that there exists a positive constant $B=\delta_0 |c|>0$ such that
\[
B \leq f^{-p}(x,t)f_t(x,t),
\]
for all $(x,t)\in M\times \mathbb{R}$. Integrating with respect to $t$ over an interval $[t_1,t_2]$ gives
\begin{equation*}
    \left.\frac{1}{1-p}f^{1-p}(x, t)\right|^{t_2}_{t_1}=\frac{1}{p-1}\big(f^{1-p}(x, t_1)-f^{1-p}(x,t_2)\big)\geq B(t_2-t_1),
\end{equation*}
which implies that
\[
\frac{1}{p-1}f^{1-p}(x,t_1)-B(t_2-t_1)\geq \frac{1}{p-1}f^{1-p}(x,t_2)>0.
\]
However, letting $t_2$ be sufficiently large leads to a contradiction, since the left-hand side tends to $-\infty$ while the right-hand side remains strictly positive. 

In case (2), suppose for contradiction that there exists a positive ancient solution $f$ to \eqref{e:main} on $M\times(-\infty, T)$. The same argument as above gives
\[
\frac{1}{1-p}f^{1-p}(x,t_2)-B(t_2-t_1)\ge \frac{1}{1-p}f^{1-p}(x,t_1)>0.
\]
By letting $t_1\rightarrow-\infty$, we get a contradiction.
This completes the proof.
\end{proof}

\begin{proof}[Proof of Corollary~\ref{thm:liouville2}]
    Suppose for contradiction that $f$ is a nontrivial positive solution to \eqref{e:main}. By assumption, $p$ satisfies (P3) and $a(x,t)$ satisfies (A2) in Theorem~\ref{thm:main2}, so that we have the Li-Yau gradient estimate
    \begin{equation*}
    \lambda\frac{|\nabla f|^2}{f^2}-\frac{f_t}{f}+\frac{\lambda}{p} a(x,t)f^{p-1}\leq C(n,p,\lambda)\left(\frac{1}{T}+\frac{1}{R^2}\right)
    \end{equation*}
    on $P_{R/2, T/2}$. 
    Letting $T\rightarrow+\infty$ and $R\rightarrow +\infty$, we deduce that for any $(x,t)\in M\times (-\infty,T_0)$, 
    \[
    \frac{\lambda}{p}a(x,t)\leq f^{-p}(x,t)f_t(x,t).
    \]
    In particular,
    \[
    \frac{\lambda}{p}a(x_0,t)\leq f^{-p}(x_0,t)f_t(x_0,t).
    \]
    Integrating over an interval $[t_1,t_2]$ gives
    \[
    \left.\frac{1}{1-p}f^{1-p}(x_0,t)\right|^{t_2}_{t_1}=\frac{1}{p-1}\big(f^{1-p}(x_0,t_1)-f^{1-p}(x_0,t_2)\big)\geq \frac{\lambda}{p}\int_{t_1}^{t_2}a(x_0,t)dt,
    \]
    which implies that
    \[
    \frac{1}{p-1}f^{1-p}(x_0,t_1)-\frac{\lambda}{p}\int_{t_1}^{t_2}a(x_0,t)dt\geq \frac{1}{p-1}f^{1-p}(x_0,t_2)>0.
    \]
    Letting $t_1=T_1$ and $t_2\rightarrow T_0$ leads to a contradiction.
\end{proof}
\begin{proof}[Proof of Corollary~\ref{thm:liouville3}]
   Since $\sup_M a>0$, there exists a point $x_0\in M$ such that
$a(x_0)>0$. If $n\geq 4$, the conclusion follows from
Theorem~\ref{thm:main2}, since $a_t=0$ and
\[
(\Delta+\partial_t)a=\Delta a\geq 0.
\]
If $n\leq 3$, Theorem~\ref{thm:main3} applies because
\[
L_\lambda a=\Delta a\geq 0.
\]
Thus, in either case, letting $T\to+\infty$ and $R\rightarrow+\infty$ in the corresponding
gradient estimate yields
\[
\frac{\lambda}{p}a(x_0)
\leq f^{-p}(x_0,t)\,f_t(x_0,t).
\]

Integrating this inequality over $[t_1,t_2]$, we obtain
 \[
    \frac{1}{p-1}f^{1-p}(x_0,t_1)-\frac{\lambda}{p}a(x_0)(t_2-t_1)\geq \frac{1}{p-1}f^{1-p}(x_0,t_2)>0.
    \]
    Letting $t_2\rightarrow +\infty$ leads to a contradiction.
\end{proof}

\section*{Acknowledgments}

The authors would like to thank Professor Youde Wang for enlightening discussions and continuous support. This work is partially 
supported by NSFC No. 12371061 and Natural Science Foundation of Fujian Province of China No. 2026J011003.


\begin{thebibliography}{99}

\bibitem{BV1998}
M.-F.~Bidaut-V\'eron, 
Initial blow-up for the solutions of a semilinear parabolic equation with source term. 
\emph{\'Equations aux d\'eriv\'ees partielles et applications}, Gauthier-Villars, \'Ed. Sci. M\'ed. Elsevier, Paris, 1998, pp. 189--198.

\bibitem{CCK15}
X.~Cao, M.~Cerenzia, and D.~Kazaras,
Harnack estimate for the endangered species equation.
\emph{Proc. Amer. Math. Soc.} \textbf{143} (2015), no. 10, 4537--4545.

\bibitem{CCM23}
D.~Castorina, G.~Catino, and C.~Mantegazza, 
Semilinear Li and Yau inequalities. 
\emph{Ann. Mat. Pura Appl.} \textbf{202} (2023), 827--850.

\bibitem{CM2021}
D.~Castorina and C.~Mantegazza, 
Ancient solutions of superlinear heat equations on Riemannian manifolds. 
\emph{Commun. Contemp. Math.} \textbf{23} (2021), no. 8, Paper No. 2050085, 16 pp.

\bibitem{CMS19}
D.~Castorina, C.~Mantegazza, and B.~Sciunzi, 
A Liouville theorem for superlinear heat equations on Riemannian manifolds. 
\emph{Milan J. Math.} \textbf{87} (2019), 303--313.

\bibitem{CDM26}
H.-J.~Chen, S.-Z.~Du, and Y.-X.~Ma,
Li-Yau inequality and Liouville property to a semilinear heat equation on Riemannian manifolds.
\emph{Manuscripta Math.} \textbf{177} (2026), 25.

\bibitem{F66}
H.~Fujita,
On the blowing up of solutions of the Cauchy problem for $u_t = \Delta u + u^{1+\alpha}$.
\emph{J. Fac. Sci. Univ. Tokyo Sect. I} \textbf{13} (1966), no. 2, 109--124.

\bibitem{GS1981}
B.~Gidas and J.~Spruck, 
Global and local behavior of positive solutions of nonlinear elliptic equations. 
\emph{Comm. Pure Appl. Math.} \textbf{34} (1981), no. 4, 525--598.

\bibitem{GK1985}
Y.~Giga and R.~V.~Kohn, 
Asymptotically self-similar blow-up of semilinear heat equations. 
\emph{Comm. Pure Appl. Math.} \textbf{38} (1985), no. 3, 297--319.

\bibitem{GK1987}
Y.~Giga and R.~V.~Kohn,
Characterizing blowup using similarity variables. 
\emph{Indiana Univ. Math. J.} \textbf{36} (1987), 1--40. 

\bibitem{GK1989}
Y.~Giga and R.~V.~Kohn, 
Nondegeneracy of blowup for semilinear heat equations. 
\emph{Comm. Pure Appl. Math.} \textbf{42} (1989), no. 6, 845--884.

\bibitem{Hamilton1993}
R.~S.~Hamilton,  
A matrix Harnack estimate for the heat equation. 
\emph{Comm. Anal. Geom.} \textbf{1} (1993), no. 1, 113--126.

\bibitem{H73}
K.~Hayakawa,
On nonexistence of global solutions of some semilinear parabolic differential equations.
\emph{Proc. Japan Acad.} \textbf{49} (1973), no. 7, 503--505.

\bibitem{HWW24}
J.~He, Y.~Wang, and G.~Wei,
Gradient estimate for solutions of the equation $\Delta_p v+av^q=0$ on a complete Riemannian manifold.
\emph{Math. Z.} \textbf{306} (2024), no. 3, Paper No. 50.

\bibitem{HW22}
P.~Huang and Y.~Wang,
Gradient estimates and Liouville theorems for Lichnerowicz equations.
\emph{Pacific J. Math.} \textbf{317} (2022), 363--386.

\bibitem{KST77}
K.~Kobayashi, T.~Sirao, and H.~Tanaka,
On the growing up problem for semilinear heat equations.
\emph{J. Math. Soc. Japan} \textbf{29} (1977), no. 3, 407--424.

\bibitem{Levine1990}
H.~A.~Levine,
The role of critical exponents in blowup theorems.
\emph{SIAM Rev.} \textbf{32} (1990), 262--288.

\bibitem{L1991}
J.~Li, 
Gradient estimates and Harnack inequalities for nonlinear parabolic and nonlinear elliptic equations on Riemannian manifolds. 
\emph{J. Funct. Anal.} \textbf{100} (1991), no. 2, 233--256.

\bibitem{LY86}
P.~Li and S.-T.~Yau,
On the parabolic kernel of the Schr\"odinger operator.
\emph{Acta Math.} \textbf{156} (1986), 153--201.

\bibitem{LZ25}
Z.~Lu and L.~Zhu,
Some new Li-Yau inequalities for semilinear parabolic equations and its applications.
\emph{J. Geom. Anal.} \textbf{35} (2025), no. 7, Paper No. 206, 31 pp.

\bibitem{MZ2026}
F.~Merle and H.~Zaag,
On degenerate blow-up profiles for the subcritical semilinear heat equation.
\emph{J. Eur. Math. Soc.} \textbf{28} (2026), no. 3, 1081--1146.

\bibitem{PWW21}
B.~Peng, Y.~Wang, and G.~Wei,
Yau type gradient estimates for $\Delta u + au (\log u)^p + bu = 0$ on Riemannian manifolds.
\emph{J. Math. Anal. Appl.} \textbf{498} (2021), no. 1, Paper No. 124963, 24 pp.

\bibitem{PQS07}
P.~Pol\'a\v{c}ik, P.~Quittner, and P.~Souplet, 
Singularity and decay estimates in superlinear problems via Liouville-type theorems. II. Parabolic equations. 
\emph{Indiana Univ. Math. J.} \textbf{56} (2007), no. 2, 879--908.

\bibitem{Q16}
P.~Quittner,
Liouville theorems for scaling invariant superlinear parabolic problems with gradient structure.
\emph{Math. Ann.} \textbf{364} (2016), 269--292.

\bibitem{Q21}
P.~Quittner,
Optimal Liouville theorems for superlinear parabolic problems.
\emph{Duke Math. J.} \textbf{170} (2021), 1113--1136.

\bibitem{QS07}
P.~Quittner and P.~Souplet,
\emph{Superlinear Parabolic Problems: Blow-up, Global Existence and Steady States}.
Birkh\"auser Verlag, Basel, 2007.

\bibitem{SC92}
L.~Saloff-Coste,
Uniformly elliptic operators on Riemannian manifolds.
\emph{J. Differential Geom.} \textbf{36} (1992), 417--450.

\bibitem{SW25}
C.~Song and J.~Wu,
Universal gradient estimates of $\Delta u+a(x)u^p(\ln(u+c))^q=0$ on complete Riemannian manifolds.
\emph{J. Differential Equations} \textbf{434} (2025), Paper No. 113257, 19 pp.

\bibitem{SZ06}
P.~Souplet and Q.~S.~Zhang,
Sharp gradient estimate and Yau's Liouville theorem for the heat equation on noncompact manifolds.
\emph{Bull. Lond. Math. Soc.} \textbf{38} (2006), no. 6, 1045--1053.

\bibitem{WjW24}
J.~Wang and Y.~Wang,
Gradient estimates for $\Delta u+a(x)u\log u+b(x)u=0$ and its parabolic counterpart under integral Ricci curvature bounds.
\emph{Comm. Anal. Geom.} \textbf{32} (2024), no. 4, 923--975.

\bibitem{WjW24*}
J.~Wang and Y.~Wang,
Boundedness and gradient estimates for solutions to $\Delta u+a(x)u\log u+b(x)u=0$ on Riemannian manifolds.
\emph{J. Differential Equations} \textbf{402} (2024), 495--517.

\bibitem{WWW2025}
K.~Wang, J.~Wei, and K.~Wu, 
F-stability, entropy and energy gap for supercritical Fujita equation. 
\emph{J. Reine Angew. Math.} \textbf{822} (2025), 49--106.

\bibitem{WW23}
Y.~Wang and G.~Wei,
On the nonexistence of positive solution to $\Delta u+au^{p+1}=0$ on Riemannian manifolds.
\emph{J. Differential Equations} \textbf{362} (2023), 74--87.

\bibitem{Y08}
Y.~Yang,
Gradient estimates for a nonlinear parabolic equation on Riemannian manifolds.
\emph{Proc. Amer. Math. Soc.} \textbf{136} (2008), no. 11, 4095--4102.

\bibitem{Zhang1998}
Q.~S.~Zhang, 
A new critical phenomenon for semilinear parabolic problems.
\emph{J. Math. Anal. Appl.} \textbf{219} (1998), no. 1, 125--139.

\bibitem{Zhang1999}
Q.~S.~Zhang, 
Blow-up results for nonlinear parabolic equations on manifolds.
\emph{Duke Math. J.} \textbf{97} (1999), no. 3, 515--539.

\bibitem{Z25}
W.~Zheng,
Gradient estimates for a weighted parabolic equation via Moser iteration.
\emph{J. Math. Anal. Appl.} \textbf{560} (2026), no. 1, Paper No. 130623.

\end{thebibliography}
\end{document}